\documentclass[reqno,a4paper]{amsart}
\usepackage{graphicx}
\usepackage{subcaption}
\usepackage{mathrsfs,amsthm,amsmath,amssymb,amsfonts,enumerate,lipsum,appendix,mathtools}
\makeatletter
\newtheorem*{rep@theorem}{\rep@title}
\newcommand{\newreptheorem}[2]{%
	\newenvironment{rep#1}[1]{%
		\def\rep@title{#2 \ref{##1}}%
		\begin{rep@theorem}}%
		{\end{rep@theorem}}}
\makeatother
\newtheorem{theorem}{Theorem}[section]
\newreptheorem{theorem}{Theorem}
\newtheorem{lemma}{Lemma}[section]

\newtheorem{corollary}[lemma]{Corollary}
\newtheorem{proposition}[lemma]{Proposition}
\newtheorem{remark}[lemma]{Remark}
\newtheorem{definition}[lemma]{Definition}

\newtheorem{observation}{Observation}
\usepackage[nointegrals]{wasysym}
\usepackage[misc]{ifsym}
\usepackage[dvipsnames]{xcolor}
\definecolor{ao}{rgb}{0.0, 0.5, 0.0}
\definecolor{lasallegreen}{rgb}{0.03, 0.47, 0.19}
\usepackage{hyperref}
\hypersetup{
	colorlinks=true,
	linkcolor=blue,
	filecolor=magenta,      
	urlcolor=cyan,
	citecolor=lasallegreen,
}
\usepackage[margin=0.75in]{geometry}

\numberwithin{equation}{section}

\let\oldnorm\norm
\def\norm{\@ifstar{\oldnorm}{\oldnorm*}}
\newcommand{\md}[1]{\left\vert #1 \right\vert}
\newcommand{\dq}{\coloneqq}

\newcommand{\RN}{{\mathbb{R}^d}}

\newcommand{\al} {\alpha}
\newcommand{\pa} {\partial}

\newcommand{\de} {\delta}
\newcommand{\De} {\Delta}
\newcommand{\Dep} {\Delta_p}
\newcommand{\ga} {\gamma}
\newcommand{\Ga} {\Gamma}

\newcommand{\Om} {\Omega}
\newcommand{\la} {\lambda}

\newcommand{\si} {\sigma}
\newcommand{\Si} {\Sigma}

\newcommand{\noi} {\noindent}
\newcommand{\SN}{\mathbb{S}^{d-1}}

\newcommand{\dx}{{\,\rm d}x}

\newcommand{\dS}{{\,\rm dS}}

\newcommand{\T}{\theta}

\newcommand{\cm}{^\mathrm{c}}
\newcommand{\C}{{\mathcal C}}

\newcommand{\F}{{\mathcal F}}

\newcommand{\R}{{\mathbb R}}

\renewcommand{\P}{{\mathcal P}}

\def\S{{\mathcal S}}

\newcommand\Item[1][]{%
	\ifx\relax#1\relax  \item \else \item[#1] \fi
	\abovedisplayskip=0pt\abovedisplayshortskip=0pt~\vspace*{-\baselineskip}}

\newcommand{\eps} {\varepsilon}

\newcommand{\wide}[1]{\widetilde{#1}}

\newcommand{\Int}{\displaystyle \int \limits }
\def\w{{\widetilde w}}

\def\dx{{\,\rm d}x}

\def\sb2{{{\mathcal D}^{1,2}_0(B_1^c)}}

\def\w2r{{{ W}^{2,2}(\R^N)}}

\def\d2{{{\mathcal D}^{2,2}_0(\Om )}}
\def\hd#1{H^1_{{\scriptscriptstyle \Ga _D^{#1}}}(\Om _{#1})}
\def\HD#1{{H^1_{{\scriptstyle \Ga _D}}(\Om _{#1})}}

\def\A{{\mathscr A}}

\def\C{{\mathcal C}}
\def\D{{\mathcal D}}

\def\H{{\mathcal H}}

\def\S{{\mathcal S}}

\def\F{{\mathcal F}}

\def\h1{{H^1(\Om)}}

\def\ws2{{\F_{\frac{N}{2}}}}

\def\c1Loc{{\C_{loc}^1}}
\usepackage[foot]{amsaddr}
\usepackage[hyperpageref]{backref}

\usepackage[pagewise, displaymath, mathlines]{lineno}
\usepackage{cite}
\title[Shape variation via the geometry of eigenfunctions]{A shape variation result via the geometry of eigenfunctions}
\author{T. V. Anoop$^{1*}$ \and K. Ashok Kumar$^{1}$}
\address{$*$ Corresponding author}
\address{$1$ Department of Mathematics, Indian Institute of Technology Madras, Chennai 600 036, India}
\email{anoop@iitm.ac.in, s.r.asoku@gmail.com}
\author{S Kesavan$^{2}$}
\address{$2$ Adjunct Professor, Department of Mathematics, Indian Institute of Technology Madras, Chennai 600036, India}
\email{kesh@imsc.res.in}
\begin{document}
	\begin{abstract}
		We discuss some of the geometric properties, such as the foliated Schwarz symmetry, the monotonicity along the axial and the affine-radial directions, of the first eigenfunctions of the Zaremba problem for the Laplace operator on annular domains. These fine geometric properties, together with the shape calculus, help us to prove that the first eigenvalue is strictly decreasing as the inner ball moves towards the boundary of the outer ball.
	\end{abstract}
	\maketitle
	\tableofcontents
	\section{Introduction}\label{intro}
This paper has two objectives. The first objective is {to} study the geometric properties, such as symmetry and monotonicity, of the first eigenfunctions of the Zaremba problem (eigenvalue problem with mixed boundary conditions)
on annular domains. The second objective is {to} study the domain variation of the first eigenvalues
over a family of annular domains. More precisely, for $d\geq 2$ 
we consider the following family of annular domains:
\noindent
\begin{equation}\label{family}\tag{$\mathcal{A}$}
\mbox{for given } 0<R_0<R_1 <\infty ,
\ \
\A _{R_0, R_1}\dq \left\{\Om _s=B_{R_1}(0)\setminus \overline{B_{R_0}(s e_1)}\subset \RN : 0\leq s<R_1-R_0 \right\},
\end{equation}
where $B_r(z)$ is the open ball centered at $z\in \RN$ with radius $r>0$ and $e_1=(1,0,\dots ,0)\in \RN .$ For $\Om _s \in \A _{R_0, R_1} $, we set the boundary $\pa \Om _s =\Ga _D\cup \Ga _N$ with
$\Ga _N= \pa B_{R_1}(0),\, \Ga _D=\pa B_{R_0}(s e_1)$ and we consider the following Zaremba problem for the Laplace operator on the annular domain $\Om _s$:
\begin{equation}\label{eigen}\tag{N-D}
\left .
\begin{aligned}
-\Delta u
&= \tau u \text{ in } \, \Omega _s,\\
u&=0 \text{ on }
\Gamma _{D},\\
\frac{\partial {u}}{\partial {n}}&=0 \text{ on } \Gamma _{N},
\end{aligned}
\right \}
\end{equation}
%
where $\tau $ is a real number and $n$ is the unit outward normal to $\Om _s$. Let $\HD{s}$ be the closure of $\C^\infty _{\Ga _D}(\Om _s) \dq \left\{\phi \in \C^\infty (\Om _s): \mathrm{supp}(\phi)\cap \Ga _D=\emptyset \right\}$ in the Sobolev space $H^1(\Om _s)$, i.e., $\HD{s}=\overline{\C^\infty _{\Ga _D} (\Om _s)}^{\|\cdot \|_{H^1(\Om _s)}}$. It is easy to observe that, the trace is zero on $\Ga _D$ for the functions in $\HD{s}$, therefore

\noindent \begin{equation*}
\HD{s}=\left\{u\in H^1(\Om _s) : \ga _0 (u)|_{\Ga _D}=0\right\} .
\end{equation*}
A real number $\tau $ is called an eigenvalue of~\eqref{eigen} if there exists $u\in \HD{s}\setminus \{0\}$ such that the following weak formulation holds:

\noindent \begin{equation*}
\int _{\Om _s}\nabla u\cdot \nabla \varphi \dx-\tau \int _{\Om _s}  u\varphi \dx =0 \quad \forall \varphi \in \HD{s} ,
\end{equation*}
and the function $u$ is called the eigenfunction corresponding to the eigenvalue $\tau $,
and the pair $(u, \tau)$ is referred as an eigenpair
of~\eqref{eigen}. Since the measure of $\Om _s$ is finite, by a standard application of the spectral theorem for self-adjoint compact operators, the set of all the eigenvalues of \eqref{eigen} is a sequence of positive real numbers tending to infinity, and the corresponding eigenfunctions form an orthonormal basis for $L^2(\Om _s)$. Let 

\noindent
\begin{equation*}
\S \dq \left\{u\in L^2(\Om _s):\int _{\Om _s}u^2\dx=1 \right\} \mbox{ and }  J(\phi ) \dq \int _{\Om _s}|\nabla \phi |^2\dx \mbox{ for }\phi \in H^1(\Om _s),
\end{equation*}
then, by using the Lagrange's multipliers {theorem}, we easily verify that the critical values of~$J$ over $\S \cap \HD{s}$ are precisely the eigenvalues of~\eqref{eigen} and the corresponding critical points of~$J$ are the eigenfunctions. The least eigenvalue $\tau _1(s)\dq \tau _1(\Om _s)$ of~\eqref{eigen} has the following variational characterization:
%

\noindent
\begin{equation*}
\tau _1(s)=\inf \left\{\int _{\Om _s}|\nabla \phi |^2 \dx : \phi \in \S \cap \HD{s} \right\}.
\end{equation*}
\noindent Using the Picone's identity and by a standard application of the maximum principle, it follows that $\tau _1(s)$ is simple (the dimension of the eigenspace is one) and the corresponding eigenfunctions have constant sign in $\Om _s$ (see \cite[Appendix A]{AAK}). Let $\la _1(s)\dq \la _1(\Om _s)$ be the first Dirichlet eigenvalue of~$-\De $ in $\Om _s$. The eigenvalue $\la _1(s)$ has the following variational characterization:

\noindent
\begin{equation*}
\la _1(s)=\inf \left\{\int _{\Om _s}|\nabla \phi |^2 \dx :\phi \in \S \cap H_0^1(\Om _s) \right\}.
\end{equation*}
Since $H^1_0(\Om _s)\subseteq \HD{s} $, we have

\noindent
\begin{equation*}
\tau _1(s)\leq \la _1(s).
\end{equation*}
The inequality above is strict, as the $(d-1)$-dimensional Hausdorff measure of $\Ga _N$ is strictly positive.
Many authors~\cite{Joseph,Kuttler,KuttlerSigillito,NAGAYA1977545,Yee} have studied methods to approximate the fundamental frequencies and corresponding wave-guides of an annular membrane ($d=2$). Also, the dependency of the fundamental frequency on the shape of the membrane has been studied in~\cite{Joseph}. For a review on the eigenvalues of the Laplace operator and the geometrical structure of the first eigenfunctions in $\R ^2$, we refer to~\cite{KuttlerSigillito,Nguyen}.
%
\medskip
\par The qualitative properties of the solutions of the semi-linear equation
\begin{equation}\label{semilinear}
-\De u=f(x,u)
\end{equation}
with Dirichlet boundary condition in symmetric domains are well studied. Gidas-Ni-Nirenberg's theorem~\cite{GidasNiNirenberg} gives the radial symmetry of the {smooth positive} solutions of~\eqref{semilinear} in whole $\RN$ as well as in the radial domains like balls or concentric annular domains. Whereas Lazer-McKenna's an abstract symmetry theorem~\cite[Theorem 1]{LaMcK88} gives the radial symmetry of the solutions of~\eqref{semilinear} with Dirichlet or Neumann condition on the radial domains. The same was studied for larger classes of non-linearities,
in~\cite{NaokiKohtaro, Weth2010, Maris2009, Bartsch-Willem05, BeNi, LiNi, Smets03}.
Perdo~and~Tobis~\cite{Pedro-Tobis} have studied the symmetry and monotonicity of the solutions of~\eqref{semilinear} with Neumann boundary condition on balls with a particular non-linearity,
$f(x,u)=\la _p |u|^{p-1}u+\mu _p$
with appropriate constants $\la _p$ and $\mu _p$. The symmetry and monotonicity of the positive solutions of the Zaremba problem of the semi-linear equation~\eqref{semilinear}, i.e., the mixed boundary problem of~\eqref{semilinear}, are studied only in the spherical cones, see~\cite{Henri-Pacella, Zhu, ChenYao, ChenYaoLi}. To the best of our knowledge, there are no results available on the geometry of the eigenfunctions of the Laplace operator for annular domains.
\par
In this article, we explore how the symmetries of annular domains are inherited by the first eigenfunctions of the Zaremba problem~\eqref{eigen}.
We have plotted the graph of an eigenfunction corresponding to the first eigenvalue $\tau _1(\Om )$  of~\eqref{eigen} for an annular domain $\Om \subset \R ^2$ in Figures~\ref{fig:test1}-\ref{fig:test3}, using \emph{Mathematica 12}. In Figures~\ref{fig:test1}-\ref{fig:test3}, 
$\Om =B_{R_1}(0)\setminus \overline{B_{R_0}(s e_1)}$ for some $s\in [0, R_1-R_0)$ and $u$ is
a non-negative eigenfunction corresponding to~$\tau _1(s)$. We observe that
\begin{enumerate}[(a)]
	\item\label{Obs:1}
	$u$ is symmetric with respect to $e_1$-axis, called as the axial symmetry (see Figure~\ref{fig:test2});
	\item\label{Obs:2}
	$u$ increases along the circles centered at the origin from right to left on both upper and lower parts (this is called the foliated Schwartz symmetry, see Figure~\ref{fig:test2});
	\item\label{Obs:3}
	$u$ increases along the straight lines {emerging} from $s e_1$ (see Figure~\ref{fig:test1});
	\item\label{Obs:4}
	$u$ decreases along the $e_1$-direction in a subdomain (see Figure~\ref{fig:test3});
	\item\label{Obs:5}
	$u$ has the unique maximum (peak) at $-R_1 e_1$, i.e., on the Neumann boundary (see Figure~\ref{fig:test1}).
\end{enumerate}
%
%
\begin{figure}[tb]\label{fig}
	\centering
	\begin{subfigure}{0.49\linewidth}
		\centering
		\includegraphics[width=0.7\linewidth]{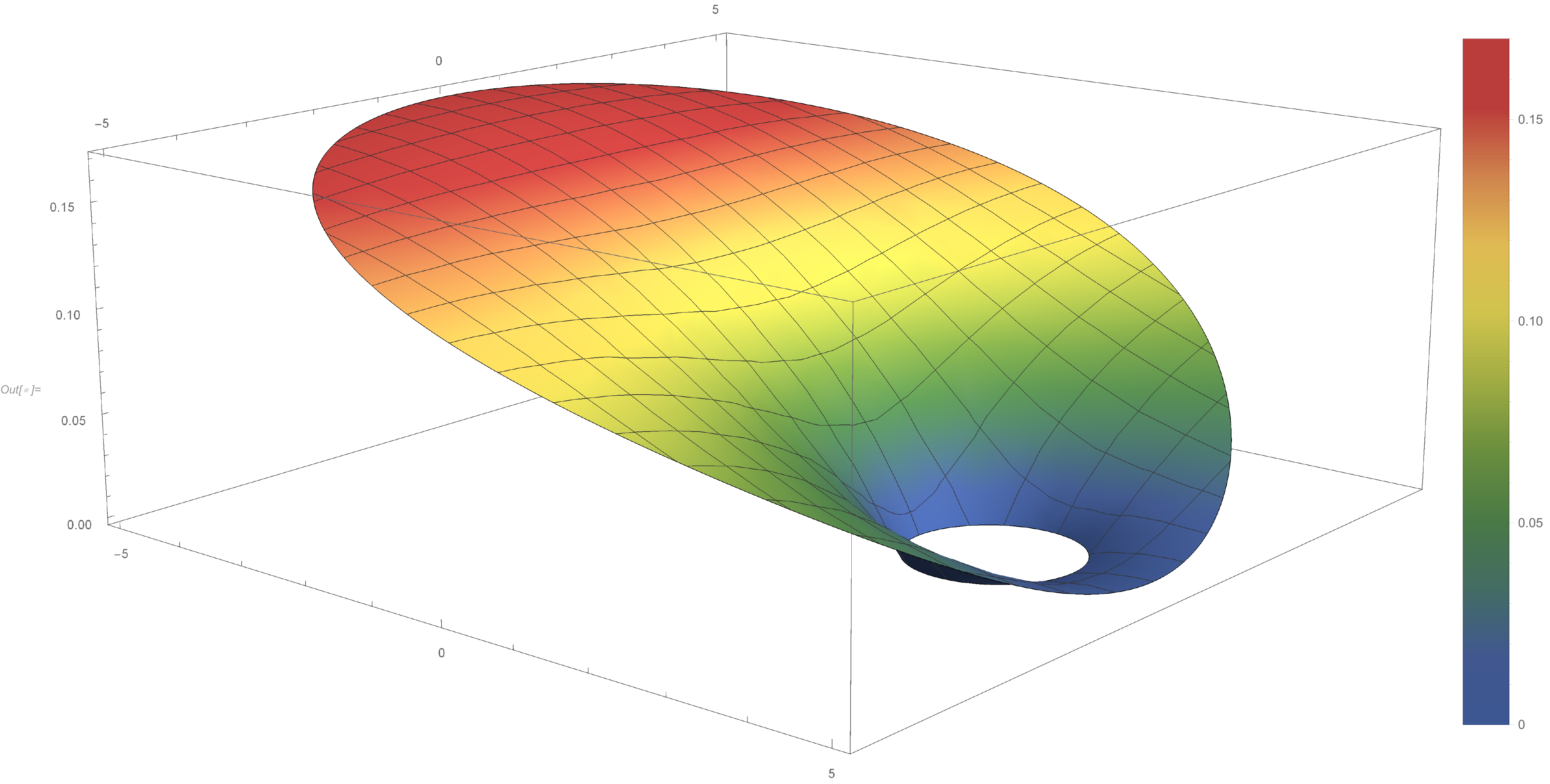}
		\subcaption{}
		\label{fig:test1}
	\end{subfigure}
	\begin{subfigure}{0.49\linewidth}
		\centering
		\includegraphics[width=0.5\linewidth]{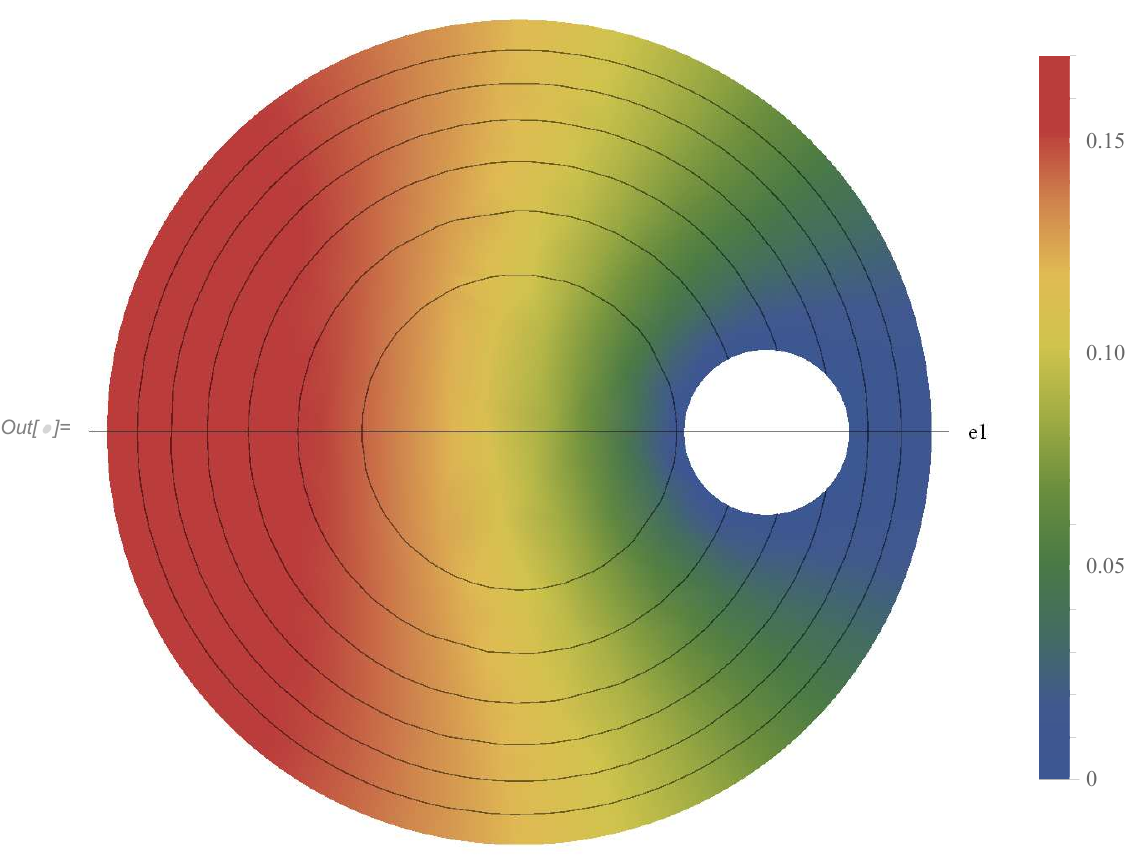}
		\subcaption{}
		\label{fig:test2}
	\end{subfigure}
	\begin{subfigure}{0.49\linewidth}
		\centering
		\includegraphics[width=0.5\linewidth]{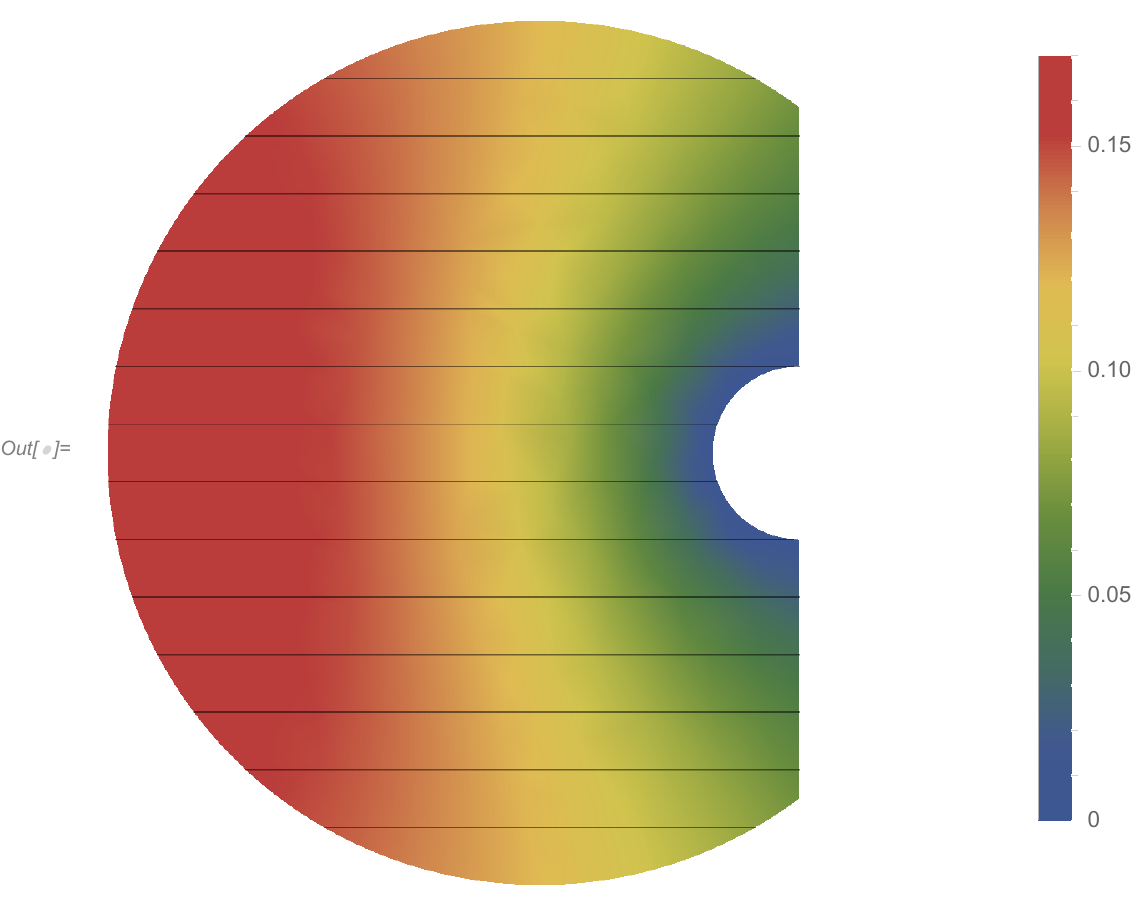}
		\subcaption{}
		\label{fig:test3}
	\end{subfigure}
	%
	\caption{\small (A) A non-negative first eigenfunction in $B_5(0)\setminus \overline{B_1(3 e_1)}\subset \R ^2$. (B) The foliated Schwartz symmetry, (C) Monotonicity along the $e_1$-direction up to $x<3$.}
\end{figure}
%
Our aim is to establish the observations above, using some analytic tools.
The observations~\eqref{Obs:1}-\eqref{Obs:5} are deduced from several theorems which have their own independent interest (Theorem~\ref{thm 0} -- \ref{thm 2}).
These theorems are combined as a unified theorem and stated below:
\begin{theorem}\label{thm uni}
	Let $\Om _s \in \A_{R_1 R_0}$ be an annular domain as given in \eqref{family} with $0<s<R_1-R_0$. Let $u$ be a positive eigenfunction of~\eqref{eigen} in $\Om _s$ corresponding to the first eigenvalue $\tau _1(s)$. Then
	\begin{enumerate}[\rm (a)]
		\item \label{thm uni1}
		$u$ has the foliated Schwarz symmetry in $\Om _s$ with respect to $-\R ^+ e_1$;
		\item \label{thm uni2}
		$u$ is strictly increasing along all the affine-radial directions from $s e_1$ in $\Om _s$, i.e., 
		
		\noindent
		\begin{equation*}
		\nabla u(x)\cdot (x-se_1)>0 \mbox{ for }x\in \Om _s\setminus \{\pm R_1 e_1\};
		\end{equation*}
		\item \label{thm uni3}
		$u$ is strictly decreasing in the $e_1$-direction on the sub-region
		$\big\{x\in \Om _s : x_1<s \big\}$ of $\Om _s$, i.e., 
		
		\noindent
		\begin{equation*}
		\displaystyle \frac{\pa u}{\pa x_1}<0\mbox{ on } \big\{x\in \Om _s : x_1<s \big\}.
		\end{equation*}
		%
	\end{enumerate}
\end{theorem}
%
\noindent {We prove this theorem using} some variants of the maximum principle and the comparison principles which are applicable to the problems with mixed boundary conditions.
%
\medskip
\par Our next objective is to study the behavior of $\tau _1 (s)$ for $s\in [0,R_1-R_0)$.
Hersch~\cite{Hersch} proved that $\tau _1$ attains its maximum only at $s=0$, when $d=2$. This result has been extended to the higher dimensions $(d\geq 2)$ by Anoop~and~Ashok~\cite{AAK}, using the method of interior parallels. However, this method fails to give any insight on 
{the ordering} of $\tau_1(s_1)$ and $\tau_1(s_2)$ for any two  non-zero $s_1,s_2\in [0,R_1-R_0).$
%
\medskip
\par For the Dirichlet problem, the behavior of $\la _1(s)$ for $s\in [0,R_1-R_0)$ {is} first considered by Ramm~and~Shivakumar~\cite{Ramm-Shivakumar}. They have conjectured that \emph{`for $d=2$, the function $\la _1(\cdot )$ is strictly decreasing on $[0,R_1-R_0)$'}, and
supported this conjuncture by a numerical evidence showing that 
$\la _1'(s)<0 \mbox{ for } 0<s<R_1-R_0;$ 
later they gave an analytic proof attributed to Ashbaugh in an arXiv paper (arxiv:math-ph/9911040).
An analytic proof for this is given by Harrell~et~al.~\cite{Harrell}, and Kesavan~\cite{Kesavan} independently for $d\geq 2$. Further, Anisa~et~al.~\cite{anisa2} have partially extended this result to a non-linear operator the $p$-Laplacian, defined as 
$-\Dep u\dq -\mathrm{div}(|\nabla u|^{p-2}\nabla u)$ 
for $1< p<\infty $,
by showing that $\la _1'(s)\leq 0$. Anoop~et~al.~\cite{Anoop18}, have proved the strict monotonicity of 
$\la _1(s)$ for the $p$-Laplacian, by showing that $\la _1'(s)<0\mbox{ for }s>0$.
The main ingredient of all the proofs is 
%
%
the representation of 
$\la _1'(s)$ as the following integral:

\noindent
\begin{equation*}
\la _1'(s)=- \Int _{x\in \pa B_{R_0}(s e_1)} \left| \frac{\partial u}{\partial n}(x) \right|^2 n_1(x) \dS ,
\end{equation*}
where $u$ is the positive eigenfunction corresponding to $\la _1(s)$ with $\|u\|_2=1$ and 
$n_1$ is the first component of the inward unit normal $n=(n_1,n_2,\dots ,n_N)$ to $B_{R_0}(s e_1)$.
This was derived using the Hadamard perturbation formula {to the shape functional $\la _1(s)$}.
%
\par

\medskip
For the first eigenvalue $\tau _1(s)$ of~\eqref{eigen}, we derive the same representation (see Section 5) for $\tau _1'(s)$ as below:

\noindent
\begin{equation*}
\tau _1'(s)= - \Int _{x\in \pa B_{R_0}(s e_1)}
\left| \frac{\partial {v}}{\partial n}(x) \right|^2 n_1(x) \dS,
\end{equation*}
where ${v}$ is the positive eigenfunction of~\eqref{eigen} corresponding to $\tau _1(s)$ with $\|{v}\|_2=1$ and
$n_1$ is the first component of the inward unit normal $n$ to $B_{R_0}(s e_1)$.
In~\cite{Harrell, Kesavan}, the authors determined the sign of $\la _1'(s)$ by ordering the values of 
$\frac{\pa u}{\pa n}(x)$ and $\frac{\pa u}{\pa n}(x^*)$ for $x\in \pa B_{R_0}(s e_1)$, where $x^*$ is the reflection of $x$ with respect to the affine hyperplane $x_1=s$. This ordering was obtained via maximum principles and comparison principles, using the ideas similar to those in the works of Serrin~\cite{Serrin}, Gidas~et~al.~\cite{GidasNiNirenberg} and Berestycki~and~Nirenberg~\cite{Berestycki}. In order to apply the comparison principle, the authors make use of the ordering $u(x)<u(x^*)$ for $x\in \pa B_{R_1}(0)$ {with $x_1>s$}. This ordering is readily available from the fact that $u$ is strictly positive in $\Om _s$. We also have a variant of the maximum and the comparison principle applicable for the Zaremba problem for the Laplce operator (see Section~\ref{Max}). Now, to determine the sign of $\tau '_1 (s)$, if we follow the arguments of~\cite{Harrell,Kesavan}, we immediately run into a trouble as the ordering of ${v(x)}$ and ${v(x^*)}$ is missing. Nevertheless, one can apply the comparison principle, if there is an ordering between ${\frac{\pa v}{\pa n}(x)}$ and ${\frac{\pa v}{\pa n}(x^*)}$. Since ${\frac{\pa v}{\pa n}}$ is zero on the outer boundary $\pa B_{R_1}(0)$ (from the Neumann condition), it is enough to determine the sign of ${\frac{\pa v}{\pa n}(x^*)}$ for $x\in \pa B_{R_1}(0)$ {with $x_1>s$}. The finer geometric properties (Theorem~\ref{thm uni}) ensure that ${\frac{\pa v}{\pa n}(x^*)}$ is positive {for $x\in \pa B_{R_1}(0)$ with $x_1>s$}. Indeed, this is the primary motivation for us to study the geometry of the eigenfunctions of the Zaremba problem in the first place.
Next, we state the monotonicity result for $\tau _1(s)$ for $s\in [0,R_1-R_0)$.
\begin{theorem}\label{thm}
	Let $\Om _s \in \A_{R_1 R_0}$ be annular domain as given in \eqref{family} with $0\leq s < R_1-R_0$. Let $\tau _1(s)$ be the first eigenvalue of \eqref{eigen} in 
	$\Om _s 
	$. 
	Then
	
	\noindent
	\begin{equation*}
	\tau _1'(0)=0 \mbox{ and } \tau _1'(s)<0 \mbox{ for } 0<s<R_1-R_0. 
	\end{equation*}
	In other words, $\tau _1(\cdot )$ is strictly decreasing on $[0,R_1-R_0).$
\end{theorem}
%
%
\par We would like to mention that the arguments and the results of the present paper can be easily extended, \emph{mutatis mutandis}, to study the torsion problem. For an annular domain $\Om _s \in \A _{R_0 R_1}$, the \emph{torsional energy} $E(s)\dq E(\Om _s)$ of $\Om _s$ is defined as

\noindent
\begin{equation*}
E(s)\dq \min\limits_{u\in H^1_{\Ga _D}(\Om _s)}\frac{1}{2}\int _{\Om _s} |\nabla u|^2\dx -\int _{\Om _s} u\dx.
\end{equation*}
The unique minimizer of $E(\Om _s)$ is called the \emph{torsion function} of $\Om _s$ and is denoted by $v_s$. The function $v_s$ is the unique weak solution of the following boundary value problem:
\begin{equation}\label{torsion}\tag{$\mathcal{T}$}
\left .
\begin{aligned} 
-\De v_s& = 1  \mbox{ in }  \Om _s,\\
v_s&= 0 \mbox{ on } \Ga _D,\\
\frac{\pa {v_s}}{\pa {n}} &= 0 \mbox{ on } \Ga _N.
\end{aligned}
\right\}
\end{equation}
The \emph{torsional rigidity} $T(s)\dq T(\Om _s)$ of~$\Om _s$ is defined by the following variational characterization:
\begin{equation}\label{tor:char1}
T(s) \dq \max \left\{\left(\int _{\Om _s} u\dx \right)^2 : u\in H^1_{\Ga _D}(\Om _s) \mbox{ with } \int _{\Om _s}|\nabla u|^2\dx=1 \right\}.
\end{equation}
We easily see that the torsional rigidity $T(s)$ can be equivalently defined as an unconstrained problem, by observing that $T(s)=-2E(s),$ i.e.,

\noindent
\begin{equation*}
T(s)=\max \limits_{u\in H^1_{\Ga _D}(\Om _s)} 2\int _{\Om _s}u\dx - \int _{\Om _s}|\nabla u|^2\dx.
\end{equation*}
Hence, the torsional rigidity $T(s)$ can be given as

\noindent
\begin{equation*}
T(s)=\int _{\Om _s}|\nabla v_s|^2\dx =\int _{\Om _s} v_s\dx.
\end{equation*}
The torsion function of the annular domain $\Om _s$ has similar geometric properties as the first eigenfunctions. We state the geometric properties of the torsion function and the monotonicity of the torsional rigidity in the following theorem, without proof.
%
\begin{theorem}\label{tor-thm uni}
	Let $\Om _s$ be an annular domain as given in~\eqref{family} with $0< s<R_1-R_0$, let $v_{s}$ be the torsion function of~$\Om _s$,
	and let $T(s)$ be the torsional rigidity of $\Om _s$. Then, the following holds:
	\begin{enumerate}[\rm (i)]
		\item \label{tor-thm uni1}
		$v_s$ has the foliated Schwarz symmetry in $\Om _s$ with respect to $-\R ^+ e_1$, where $\R^+ =[0,\infty)$,
		\item \label{tor-thm uni2}
		$v_s$ is strictly increasing along all the affine-radial directions from $s e_1$ in $\Om _s$, i.e.,
		
		\noindent
		\begin{equation*}
		\nabla v_s(x)\cdot (x-se_1)>0 \mbox{ for }x\in \Om _s\setminus \{\pm R_1 e_1\},
		\end{equation*}
		\item \label{tor-thm uni3}
		$v_s$ is strictly decreasing in the $e_1$-direction on the sub-region
		$\big\{x\in \Om _s : x_1<s \big\}$ of $\Om _s$, i.e., 
		\[\displaystyle \frac{\pa v_s}{\pa x_1}<0\mbox{ on } \big\{x\in \Om _s : x_1<s \big\},\]
		\item \label{tor-thm uni4}
		\noindent
		$T'(s)=\displaystyle \int _{\Ga _D} \left|\frac{\pa v_s}{\pa n}(x)\right|^2 n_1(x) \dS ,$
		\item $T'(0)=0 \mbox{ and } T'(s)>0\ \mbox{ for } 0<s<R_1-R_0.$ In other words, $T(\cdot )$ is strictly increasing on $[0,R_1-R_0).$
		%
		%
	\end{enumerate}
\end{theorem}
\medskip
\par The remainder of the paper is organized as follows. The foliated Schwarz symmetrization and its characterization via polarizations are introduced in Section~\ref{Symm}. Some variants of the maximum and the comparison principles for mixed boundary problems are given in Section~\ref{Max}. The foliated Schwarz symmetry and the monotonicity of the first eigenfunctions are proved in Section~\ref{Symm and Mono}. The Hadamard perturbation (the shape derivative) formula of the first eigenvalue is derived in Appendix~\ref{Shape}. In Section~\ref{Mono}, the monotonicity of the first eigenvalue is proved.
In the last section (Section~\ref{remarks}), we remark on the geometry of the other boundary problems and state a few open problems related to the shape monotonicity of the corresponding eigenvalues.
\section{Polarizations and foliated Schwarz symmetrization}\label{Symm}
In this section, we {define} of the foliated Schwarz symmetrization of functions defined on an annular domain and its characterization via polarizations. We note that, throughout the article any $x\in \RN $ is treated as a row vector $x=(x_1,x_2,\dots ,x_d)$ and therefore the standard inner product of $x,y\in \RN$ is given by $x\cdot y=x y^T=yx^T$. 
\subsection{Polarizations} 
A \emph{polarizer} is a closed affine half-space of $\RN$, and the set of polarizers is denoted by $\H$. For $H\in \H$, the boundary $\pa H$ is an affine hyperplane in $\RN$ and $H$ also gets an orientation from this affine hyperplane. 
We notice that, for $H\in \H $ 
there exists $h\in \SN \mbox{ and } b\in \RN$ such that

\noindent
\begin{equation*}
	H=\{x\in \RN : h\cdot (x-b) \leq 0\}.
\end{equation*}
Let $\si _H$ denote the reflection with respect to $\pa H$. Then, we have

\noindent
\begin{equation*}
	\si _H(x)=x-(2 h\cdot (x-b))h =x \big[I-2h^Th\big]+(2h b^T)h\quad \forall x\in \RN.
\end{equation*}
Let us denote the matrix $I-2h^T h$ by $\si _H$, then 

\noindent
\begin{equation*}
	\si _H(x)=x\si _H+(2h b^T)h\quad \forall x\in \RN.
\end{equation*}
%
%
Now, we {define} the polarization of the functions defined on balls or  annular domains. 
Let $e_1=(1,0,\dots ,0)\in \RN$. We consider the following subsets of $\H$:

\noindent
\begin{equation*}
	\H _0=\Big\{H\in \H:0\in \pa H \Big\}\mbox{ and } \H_*=\Big\{H\in \H_0:-e_1 \in H\Big\}.
\end{equation*}
\begin{definition}[Polarization]\label{polarization}
	Let $H\in \H_0$ be a polarizer and $B_{R}(0)$ be a ball, $R>0$. The polarization of a function $u:B_{R}(0)\longrightarrow \R$ with respect to $H$ is defined by

	\noindent
	\begin{equation*}
		u^H(x)\dq
		\left\{
		\begin{aligned}
			\max \left\{u(x),u(\si _H(x))\right\},& \mbox{ if } x\in B_{R}(0)\cap H,\\
			\min \left\{u(x),u(\si _H(x))\right\},& \mbox{ if } x\in B_{R}(0)\cap H\cm .
		\end{aligned}
	    \right.
	\end{equation*}
    Let $\Om _s=B_{R}(0)\setminus \overline{B_{r}(se_1)}$ be an annular domain, for $0<r<R<\infty \mbox{ and }0\leq s<R-r$. Let $u:\Om _s\longrightarrow \R $ be a non-negative function and let $\widetilde{u}:B_{R}(0)\longrightarrow \R $ be the zero extension of $u$ to the ball $B_{R}(0).$ The polarization (with respect to $H$) of $u$ is defined as the restriction of the polarization of the zero extension $\wide{u}:B_{R}(0)\longrightarrow \R$ to $\Om _s$ and it is denoted by $u^H$, i.e., 
	$u^H=\left . {\widetilde{u}}^H\right |_{\Om _s}$.
\end{definition}
%
\begin{remark}\label{rmk:fSS}
		Observe that, if $H\in \H_*$ and $u:\Om _s\longrightarrow \R$ is non-negative then
		$\wide{u}^H(x)=0$ for $x\in B_{r}(se_1)$. Therefore, for any $H\in \H _*$ we have $u^H(x)=u(x)$ for $x\in \Om _s\cap \si _H(B_{r}(se_1))$. 
\end{remark}
\medskip
For the definition of the polarization of a function defined on $\RN$ or on symmetric regions of $\RN$, we refer to \cite[Section 2]{Smets03} and \cite[Section 3]{Willem08}.
\noindent In the same spirit of \cite[Lemma 5.3]{BrockSolynin2000} and \cite[Proposition 2.3]{Smets03}, we have the following key proposition which follows from \cite[Proposition 2.18]{VS06tmna}. 
\begin{proposition}\label{propo:polarizer}
	Let $H\in \H_*$ be a polarizer 
	and $\Om _s\in \A_{R_0 R_1}$ be an annular domain as given in~\eqref{family}. Let 
	$1\leq p\leq \infty$ and $u:\Om _s\longrightarrow \R$
	be a non-negative function, and let $\widetilde{u}:B_{R_1}(0)\longrightarrow \R $ be the zero extension of $u$ to $B_{R_1}(0)$. If 
	$\wide{u}\in W^{1,p}(B_{R_1}(0))$ then
	$\wide{u}^H\in W^{1,p}(B_{R_1}(0))$ and
	
	\noindent
	\begin{equation*}
		\left\|\wide{u}^H\right\|_p=\left\|\wide{u}\right\|_p,\ \left\|\nabla \wide{u}^H\right\|_p
		=\left\|\nabla \wide{u}\right\|_p.
	\end{equation*}
	Furthermore, if $u\in W^{1,p}(\Om _s)$ with $\left. u\right |_{\Ga _D }=0$ then
	$u^H\in W^{1,p}(\Om _s)$ with $\left. u^H \right |_{\Ga _D}=0$, 
	and hence
	
	\noindent
	\begin{equation*}
		\left\|u^H\right\|_p=\left\|u\right\|_p,\ \left\|\nabla u^H\right\|_p=\left\|\nabla u\right\|_p.
	\end{equation*}
\end{proposition}
\subsection{The foliated Schwarz symmetrization}
Now, we give the definition of the foliated Schwarz symmetrization for functions defined on balls or an annular domains. 
	\begin{definition}[Foliated Schwarz symmetrization for radially symmetric domains] \label{fSS}
	\noindent
	%
		%
		Let $\R ^+ =[0,\infty)$, $r,R\in \R ^+$ with $r<R$, $a\in \RN$, and $e_1=(1,0,\dots ,0)\in \RN$. Let $\Om =B_{R}(a)\setminus \overline{B_{r}(a)}$ be a ball or a concentric annulus and let $u$ be a non-negative measurable function on $\Om $. The foliated Schwarz symmetrization of $u$ with respect to the $1$-dimensional closed half subspace $a-\R ^+ e_1$, 
		is the unique function $u^*$ such that, for $t\in (r,R)$ and $c\ge 0$
		
		\noindent
		\begin{equation*}
			\{u^*\geq c\}\cap \pa B_t(a)=B_{\rho }(a-t e_1) \cap \pa B_t(a),
		\end{equation*}
		where $\rho \geq 0$ is defined by
		
		\noindent
		\begin{equation*}
			\H^{d-1}(B_{\rho }(a-t e_1)\cap \pa B_t(a))=\H^{d-1}(\{u\geq c\}\cap \pa B_t(a)),
		\end{equation*}
		where $\H^{d-1}$ is the $(d-1)-$dimensional Hausdorff measure.
		%
		%
%
\end{definition}
\medskip
Now, we extend the definition of the foliated Schwarz symmetrization for the functions defined on non-concentric annular domains. To a non-concentric annular domain $B_{R}(0)\setminus \overline{B_{r}(s e_1)}$, one can associate two natural radial domains. One is the ball $B_{R}(0)$ and the other  is the concentric annular domain $B_{R+s}(s e_1)\setminus \overline{B_{r}(s e_1)}$. Each of these radial domains gives a foliated Schwarz symmetrization (different from the other) for a function a defined on $B_{R}(0)\setminus \overline{B_{r}(s e_1)}$.
%
\begin{definition}[Foliated Schwarz symmetrization for non-radial domains]
	Let $\Om _s=B_{R}(0)\setminus \overline{B_{r}(s e_1)} $ 
	be an annular domain, for $0<r<R<\infty, \ 0\leq s<R-r$, and let $u:\Om _s\longrightarrow \R$ be a non-negative measurable function. Let $\wide{u}$ be the zero extension of $u$ to the ball $B_{R}(0)$ and let $\overline{u}$ be the zero extension of $u$ to the concentric annular domain $B_{R +s}(s e_1)\setminus \overline{B_{r}(s e_1)}$. 
	\begin{enumerate}[\rm (a)]
		\item The foliated Schwarz symmetrization of $u$ with respect to the $1$-dimensional closed half subspace $-\R ^+ e_1$ is defined as the restriction of
		$(\wide{u})^*$ to ${\Om _s}$.
		\item The foliated Schwarz symmetrization of $u$ with respect to the $1$-dimensional closed half subspace $s e_1-\R ^+ e_1$ is defined as the restriction of
		$(\overline{u})^*$ to ${\Om _s}$.
	\end{enumerate}
	We denote both the foliated symmetrizations by $u^*$ and the distinction will be clear from the context.
\end{definition}
\begin{remark}
	$\rm (a)$ For any subset $A$ of a sphere $\pa B_t(a),\, t\in (r, R)$, the rearrangement $A^*\subset \pa B_t(a)$ is defined as $A^*=B_{\rho }(a-t e_1)\cap \pa B_t(a)$, where $\rho \geq 0$ is chosen such that 
	
	\noindent
	\begin{equation*}
		\H^{d-1}(B_{\rho }(a-r e_1)\cap \pa B_t(a))=\H^{d-1}(\{u\geq c\}\cap \pa B_t(a)).
	\end{equation*}
	Now, the foliated Schwarz symmetrization $u^*$ of a non-negative measurable function $u: B_{R}(a)\setminus \overline{B_r(a)}\longrightarrow \R $ can be given as
	
	\noindent
	\begin{equation*}
		u^*(x)=\sup \left\{c>0 : x\in \big(\{u\geq c\}\cap \pa B_r(a)\big)^* \right\}, \ x\in B_{R}(a)\setminus \overline{B_r(a)},
	\end{equation*}
	where $r=|x|$.\\
	$\rm (b)$ From the definition above, we observe that: for a non-negative function $u:\Om _s \longrightarrow \R$ we have
	$\left(\wide{u}\right)^*=0$ in $B_{R_0}(se_1).$\\
	$\rm (c)$ The foliated Schwarz symmetrization with respect to an $1$-dimensional closed affine half subspace $F$ of $\RN$ is the cap symmetrization with respect to $F$ (see \cite[Section 3]{Willem08}).
\end{remark}
\begin{definition}
	We say that a non-negative function $u:\Om _s \longrightarrow \R$ has the foliated Schwarz symmetry in $\Om _s$ with respect to the $1$-dimensional closed affine half subspace
	$a-\R ^+ e_1,$  $a\in \{0, s e_1\}$, if 
	
	\noindent
	\begin{equation*}
		u^*=u \mbox{ in } \Om _s.
	\end{equation*} 
\end{definition}
The following characterization of the foliated Schwarz symmetry of a function define on $\Om _s$ is motivated from \cite[Section 3]{Willem08} and in \cite[Section 6]{BrockSolynin2000}.
\begin{proposition}\label{propo:characterization}
	Let $\Om _s \in \A_{R_1 R_0}$ be an annular domain as given in~\eqref{family} with $0\leq s <R_1-R_0$. Let $u:\Om _s\longrightarrow \R$ be a non-negative function.
	Then 
	$u$ has the foliated Schwarz symmetry in $\Om _s$
	\begin{enumerate}[\rm (i)]
		\item with respect to the closed affine half subspace $-\R ^+ e_1$ if and only if $u^H=u \mbox{ a.e. in }\Om _s$ for every $H\in \H_*$.
		\item with respect to the closed affine half subspace $s e_1-\R ^+ e_1$ if and only if $u^H=u \mbox{ a.e. in }\Om _s$ for every $H\in s e_1+\H_*$.
	\end{enumerate}
\end{proposition}
\begin{proof}
	The characterization holds for the extensions $\wide{u}$ and $\overline{u}$ to the ball $B_{R_1}(0)$ and the concentric annulus $B_{R_1+s}(s e_1)\setminus \overline{B_{R_0}(s e_1)}$ respectively, 
	and hence it holds for $u$ in~$\Om _{s}$.
\end{proof}
The following characterization of the foliated Schwarz symmetry also follows by a reasoning along the same lines as in \cite{DamascelliPacella, Li-Zhang, Weth2010}.
\begin{remark}\label{rmk:fSS1}
	Let $a\in \{0, s e_1 \}$. A non-negative function $u:\Om _s\longrightarrow \R$ has the foliated Schwarz symmetry in $\Om _s$ with respect to $a-\R ^+ e_1$ if and only if $u$ depends only on $r=|x|$ and the polar angle
	$\T (x)=\cos ^{-1}\left(-\frac{x-a}{|x-a|}\cdot e_1\right)$,
	and $u$ is non-increasing in $\T $ for any $r>0$.
\end{remark}
	\section{Maximum principles}\label{Max}
In this section, we state the strong maximum principle and Hopf's lemma for the Laplace operator. We derive the maximum principle and the comparison principle, which hold for the Zaremba problem for the Laplace operator. Firstly, we state a regularity result for the eigenfunctions of the Zaremba problems, for a  proof see \cite[Theorem 9.19]{GilbargTrudinger}.
\begin{proposition}
    Let $\Om=\Om_s$ be a domain as given in \eqref{family} and $u\in H^2(\Om)$ be a solution of $-\De u= \la u$ for some $\la \in \R$. Then  $u\in \C^2(\Om)\cap \C^1(\overline{\Om}).$
\end{proposition}
 Next we state the strong maximum principle \cite[Theorem 2.7]{Han} and Hopf's lemma \cite[Lemma 3.4]{GilTru}.
\begin{proposition}\label{SMP}
	Let $\Om \subset \RN $ be a bounded domain. Let
	$u\in \C^2(\Om )\cap \C^1(\overline{\Om })$
	be a non-negative function satisfying
	
	\noindent
	\begin{equation*}
		-\De u\geq 0\mbox{ in }\Om .
	\end{equation*}
	Then
	\begin{enumerate}[\rm (a)]
		\item {Strong Maximum Principle:}
		$u\equiv C$ a constant, or else $u>0$ in $\Om $.
		\item {Hopf's lemma:}
		Let $x_0\in \pa \Om $ satisfying the interior sphere condition and $u(x)\geq u(x_0)$ in $\Om $.
		If $u$ is not constant in $\Om $, then

		\noindent
		\begin{equation*}
			\frac{\pa u}{\pa \nu }(x_0)<0,
		\end{equation*}
		for any outward direction $\nu $ to $\Om $ at $x_0$.
	\end{enumerate}
\end{proposition}
Next we prove a variant of the weak maximum principle for a perturbation of the Laplace operator. 
\begin{proposition}\label{aux:smp}
	Let $\Om $ be a domain in $\RN$ and $\la _1(\Om )$ be the first Dirichlet eigenvalue of $-\De$ in~$\Om$. Let $\Om ' \subseteq \Om $ be an open set, $\mu \leq \la _1 (\Om )$ and let $v\in \C^2(\Om ')\cap \C^1(\overline{\Om '})$ be a function satisfying:
	\begin{equation}\label{eqn:aux}
	\left .
	\begin{aligned}
	-\De v  - \mu v &\geq 0 \mbox{ in }\Om ',\\
	v &\geq 0  \mbox{ on }\pa \Om '.
	\end{aligned}
	\right \}
	\end{equation}
	%
	%
	If ${\rm (i)}\ \mu < \la _1 (\Om )$ or ${\rm (ii)}\ \mu =\la _1 (\Om )$ and $|\Om \setminus \Om '|> 0$, then $v\geq 0$ in $\Om '$.
	%
\end{proposition}
\begin{proof} 
	%
	Suppose $v^-\neq 0$ in $\Om '$. By multiplying~\eqref{eqn:aux} with $v^-$ and integrating by parts, we obtain
	\begin{equation}\label{aux:ibp}
	-\int _{\Om '}|\nabla v^-|^2\dx
	-\int _{\pa \Om '}\frac{\pa v}{\pa n}v^-\dS
	+ \mu \int _{\Om '} |v^-|^2\dx \geq 0.
	\end{equation} 
	
	\noindent Since $v\geq 0$ on $\pa \Om '$ we have $v^-=0$ on $\pa \Om '$ and hence from~\eqref{aux:ibp} we have
	
	\noindent
	\begin{equation*}
	\displaystyle \int _{\Om '}|\nabla v^-|^2\dx \leq \mu \int _{\Om '} |v^-|^2\dx .
	\end{equation*}
	As $v^-=0$ on $\pa \Om '$ the zero extension (denoted by $\wide{v}$) of $v^-$ to $\Om $ belongs to $H^1_0(\Om )$. Thus, from the inequality above we have
	\begin{equation}\label{aux:eqn-wmp1}
	\displaystyle \int _{\Om }|\nabla \wide{v}|^2\dx \leq \mu \int _{\Om } |\wide{v}|^2\dx .
	\end{equation}
	Therefore, from the variational characterization of~$\la _1(\Om )$ 
	we obtain $\la _1(\Om )\leq \mu.$
%
	
	\noindent
	For $\mu < \la _1(\Om )$, this immediately gives a contradiction, and hence $v\geq 0$ in $\Om '$. On the other hand, for $\mu = \la _1(\Om )$ and $|\Om \setminus \Om '|>0$, the inequality~\eqref{aux:eqn-wmp1} implies that $\wide{v}$ is an eigenfunction corresponding to $\la _1(\Om )$. This is a contradiction, as the extension $\wide{v}$ vanishes on the set $\Om \setminus \Om '$ of positive measure in $\Om $. Therefore, we have $v\geq 0$ in $\Om '$.
\end{proof}
%
%
Next, for a general domain (not necessarily an annular domain), we consider the Zaremba eigenvalue problem. For $\Om \subseteq \RN$ with
$\pa \Om =\Ga _N \sqcup \Ga _D$, 
let us consider 
the following  Zaremba eigenvalue problem on $\Om $:
\begin{equation}\label{mixed}
\left .
\begin{aligned}
-\De u&= \tau u \text{ in }\Om ,\\
u&=0 \text{ on }\Ga _D,\\
\frac{\pa {u}}{\pa {n}}&=0 \text{ on }\Ga _N.
\end{aligned}
\right \}
\end{equation}
As before, the first eigenvalue of \eqref{mixed} has the following variational characterization:

\noindent
\begin{equation*}
\tau _1(\Om )=\inf \left\{\mu : \int _{\Om }|\nabla \phi |^2\dx\leq  \mu \int _{\Om }\phi ^2 \dx
\quad \forall \phi \in \HD{}
\right\},
\end{equation*}
and the minimizer is an eigenfunction corresponding to $\tau _1(\Om )$.
%
Next we prove an analogue of Proposition~\ref{aux:smp} for $\tau_1(\Om )$.
\begin{proposition}\label{aux:wcp}
	Let $\Om $ be a domain in $\RN$ and $\tau _1(\Om )>0$ be the first eigenvalue of~\eqref{mixed} in $\Om $.
	Let $\Om '\subseteq \Om $ be an open set with  
	$\pa \Om '=\Ga _N'\sqcup \Ga _D'$ such that $\Ga _N' \subseteq \Ga _N$. Let $\mu \leq \tau _1(\Om )$ and let $v\in \C^2(\Om ')\cap \C^1(\overline{\Om '})$ satisfies:
	\begin{equation}\label{eqn1:aux}
	\left.
	\begin{aligned}
	-\De v  - \mu v &\geq 0  \mbox{ in }\Om ',\\
	v &\geq 0  \mbox{ on }\Ga _D',\\
	\frac{\pa {v}}{\pa {n}}&\geq 0 \mbox{ on }\Ga _N'.
	\end{aligned}
	\right\}
	\end{equation}
	If {\rm (i)} $\mu < \tau _1(\Om ) $ or {\rm (ii)} $\mu = \tau _1(\Om ) $ and $|\Om \setminus \Om '|> 0$, then $v\geq 0$ in $\Om '$.
\end{proposition}
\begin{proof}
	Suppose $v^-\neq 0$ in $\Om '$. By multiplying~\eqref{eqn1:aux} with $v^-$ and integrating by parts, we obtain that 
	\begin{equation}\label{aux:ibp1}
	-\int _{\Om '}|\nabla v^-|^2\dx
	-\int _{\pa \Om '}\frac{\pa v}{\pa n}v^-\dS
	+\mu \int _{\Om '} |v^-|^2\dx \geq 0.
	\end{equation}
	Since $v\geq 0$ on $\Ga _D'$,	we have $v^-=0$ on $\Ga _D'$. Then we get that
	$\int _{\pa \Om '}\frac{\pa v}{\pa n}v^-\dS \geq 0$, since 
	$\displaystyle \frac{\pa v}{\pa n}\geq 0$ on $\Ga _N'$. Hence from~\eqref{aux:ibp1} we have 
	
	\noindent
	\begin{equation*}
	\int _{\Om '}|\nabla v^-|^2\dx\leq \mu \int _{\Om '} |v^-|^2\dx .
	\end{equation*}
	As $v^-=0$ on $\Ga _D'$ the zero extension (denoted by $\wide{v}$) of $v^-$ to $\Om $ belongs to $H^1_{\Ga _D}(\Om )$. Thus, from the inequality above we have
	\begin{equation}\label{eqn:aux2}
	\displaystyle \int _{\Om }|\nabla \wide{v}|^2\dx \leq \mu \int _{\Om } |\wide{v}|^2\dx .
	\end{equation}
	Therefore, from the variational characterization 
	of $\tau _1(\Om )$ 
	we obtain $\tau _1(\Om )\leq \mu .$
	
	\noindent
	For $\mu < \tau _1(\Om )$, this immediately gives a contradiction, and hence $v\geq 0$ in $\Om '$. On the other hand, for $\mu = \tau _1(\Om )$ and $|\Om \setminus \Om '|>0$, the inequality~\eqref{eqn:aux2} implies that $\wide{v}$ is an eigenfunction corresponding to $\tau _1(\Om )$. This is a contradiction to the strong maximum principle (part (a) of Proposition~\ref{SMP}), as the extension $\wide{v}$ vanishes on the set $\Om \setminus \Om '$ of positive measure in $\Om $. Therefore, we have $v\geq 0$ in $\Om '$.
	%
	%
	%
	%
	%
\end{proof}
%
\begin{corollary}\label{coro 2}
	Let $\Om , \Om ', \tau _1 (\Om ), \mu $ and $v\in C^2(\Om ')\cap C^1(\overline{ \Om '})$ be defined as in Proposition~\ref{aux:wcp}, which satisfies~\eqref{eqn1:aux}. Further, assume that $v\geq 0$ in $\Om '$.	
	\begin{enumerate}[\rm (a)]
		\item If $\frac{\pa v}{\pa n}(x_0)>0$ for some $x_0 \in \Ga '_N$ satisfying the interior sphere condition, then $v>0$ in $\Om '$.
		\item Let $a\in \Ga '_N$ satisfy the interior sphere condition. If $v>0$ in $\Om '$, then $v(a)>0.$
	\end{enumerate}
\end{corollary}
\begin{proof}
	The proof follows easily from Hopf's  lemma (Proposition~\ref{SMP}).
	%
	%
	%
\end{proof}
\section{Symmetry and monotonicity of the eigenfunctions}\label{Symm and Mono}
In this section, we discuss the symmetry and monotonicity properties of the eigenfunctions of~\eqref{eigen} in $\Om _s$. 
Firstly, we discuss the foliated Schwarz symmetry, and secondly we discuss the monotonicity along the affine radial directions. Lastly, we discuss the monotonicity along the axial direction in a subdomain of $\Om _s$.
\subsection{Foliated Schwarz symmetry}\label{fSS-4.1}
We recall that
$\H _* = \{H\in \H : 0 \in \pa H \mbox{ and } -e_1 \in H\}.$
For $H\in \H _*$, there exists $h \in \SN $ with~$h \cdot e_1 > 0$ such that $H
=\{x \in \RN: h \cdot x \leq 0 \}$. Then, the reflection $\si _H$ with respect to $\pa H$ of 
$H\in \H_*$ is given by 

\noindent
\begin{equation*}
\si _H(x)=x \si _H\quad \forall x\in \RN,
\end{equation*}
where $\si _H$ is the matrix $I-2h^T h$.
%
%
Now, we prove the foliated Schwarz symmetry of the first eigenfunctions.
\begin{theorem}\label{thm 0}
	Let $\Om _s \in \A_{R_1 R_0}$ be an annular domain as given in~\eqref{family} with $0 < s< R_1-R_0$. Let $u>0$ be an eigenfunction of~\eqref{eigen} corresponding to the first eigenvalue $\tau _1(s)$ in 
	$\Om _s
	$. Then $u$ has the foliated Schwarz symmetry in $\Om _s$ with respect to $-\R ^+ e_1$. 
\end{theorem}
\begin{proof}
	We observe that, from Proposition~\ref{propo:polarizer}, the $L^2$-norm and the Dirichlet energy of a first eigenfunction $u$ of~\eqref{eigen} are invariant under the polarization, i.e.,
	$\left\|u^H\right\|_2 =\left\|u\right\|_2,\ \left\|\nabla u^H\right\|_2 =\left\|\nabla u\right\|_2.$
	Therefore, by the variational characterization of $\tau _1 (s)$, $u^H$ is also a first eigenfunction. Since both $u$ and $u^H$ are non-negative and have same $L^2$-norm, and by the simplicity of~$\tau _1(s)$ we obtain that $u^H=u$ in $\Om $ for any $H\in \H _*$. Therefore, by the characterization given in Proposition~\ref{propo:characterization}, $u$ has the foliated Schwarz symmetry in $\Om _s$ with respect to $-\R ^+ e_1$.
\end{proof}
%
\begin{remark}
	From the characterization in Remark~\ref{rmk:fSS1}, the foliated Schwarz symmetry of a positive first eigenfunction $u$ is equivalent to $\frac{\pa u}{\pa \T}(r,\T)\leq 0$, for $0<r\leq R_1$ and for all of the possible values of $\T$ in $\Om _s$.
\end{remark}
\par 
The next proposition gives an another proof for the foliated Schawarz symmetry of the first eigenfunctions of~\eqref{eigen}. More importantly, this gives a better insight on the geometry of the eigenfunctions. Let $H\in \H^*$ be a fixed polarizer, then for any set $A\subseteq \RN,$ we denote $A \cap H$ by $A^+$ and $A\cap H\cm $ by $A^-$.
\begin{proposition}\label{propo:fss}
	Let $\Om _s \in \A_{R_1 R_0}$ be an annular domain as given in~\eqref{family} with $0 < s< R_1-R_0$. Let $u>0$ be an eigenfunction of~\eqref{eigen} corresponding to the first eigenvalue $\tau _1(s)$ in 
	$\Om _s$. Then for $H\in \H_*$,  
	%
	%
	
	\noindent
	\begin{equation*}
	u(x)<u(\si _H(x))\mbox{ for }x\in \overline{\Om _s^-}\setminus \pa H,
	\end{equation*}
	and for an inward normal $h$  to $H,$ \noindent
	\begin{equation*}
	\frac{\pa u}{\pa h}(a)>0\mbox{ for }a\in 
	\overline{\Om _s} \cap \pa H.
	\end{equation*}
	%
	%
	%
\end{proposition}
\begin{proof}
	%
	Let $w(x)\dq u(\si _H(x)) - u(x)$  for  $x\in \overline{\Om _s ^-}$. 
	Then $w\in H^1(\Om _s^-)$
	and it satisfies the following boundary value problem: 
	
	\noindent
	\begin{equation*}
	\left.
	\begin{aligned}
	-\De w - \tau _1(s) w &=0 \mbox{ in } \Om _s^-,\\
	w = 0  \mbox{ on } \overline{ \Om _s}\cap \pa H,&\;\;
	w> 0 \mbox{ on } \Ga _D^-,\\
	\frac{\pa w}{\pa n} &= 0  \mbox{ on } \Ga _N^-.
	\end{aligned}\right\}
	\end{equation*}
	%
	%
	Now, from Proposition~\ref{aux:wcp} we obtain $w\geq 0$ in $\Om _s^-$, and then the strong maximum principle (Proposition~\ref{SMP}) implies $w>0$ in $\Om _s^-$. Therefore,  by Proposition~\ref{propo:Hopf}, $w>0$ on $\Om _s ^-\cup \Ga _N^-$. Hence $u(x)<u(\si _H(x))$ for $x\in \overline{\Om _s^-}\setminus \pa H$.
	Let $h$ be an inward normal to $H$ and $a\in \overline{\Om _s}\cap \pa H \subset  \pa(\Om _s^-).$ Observe that $\Om _s^-$  satisfies the interior sphere condition at $a \in \Om _s \cap \pa H$, $h$ is an outward normal to $\Om _s^-$ at $a$ and $w(a)=0$. Thus by Hopf's lemma, we have $\frac{\pa w}{\pa h}(a)<0$ for $a\in \Om _s \cap \pa H$ and for an inward normal $h$ to $H$. On the other hand, $h$ is tangential to $\pa \Om _s$ at $a \in \Ga _N \cap \pa H.$ Therefore, Proposition~\ref{propo:Hopf} implies that $\frac{\pa w}{\pa h}(a)<0$ for $a\in \pa \Ga _N \cap \pa H$. Notice that, $\si _H(h)=-h$ and $\si _H(a)=a$ for~$a\in \big(\Om _s \cup \Ga _N\big)\cap \pa H$. Thus 
	
	\noindent
	\begin{equation*}
	\frac{\pa w}{\pa h}(a)= \nabla u(\si _H(a))\si _H \cdot h -\nabla u(a)\cdot h = \nabla u(a) \cdot (\si _H(h)-h)=-2\nabla u(a)\cdot h .
	\end{equation*}
	Hence, we have 
	
	\noindent
	\begin{equation*}
	\frac{\pa u}{\pa h} =-\frac{1}{2}\frac{\pa w}{\pa h}>0 \mbox{ on } \big(\Om _s \cup \Ga _N \big) \cap \pa H.
	\end{equation*}
	At $a\in \Ga _D \cap \pa H$, we have $(a-s e_1)\cdot h =-sh_1>0,$ since $h_1<0$ for an inward normal $h$ to $H$. So $h$ is an inward direction to $\Om _s$ at $a$. Therefore, by Hopf's lemma, we obtain
	 
	\noindent
	\begin{equation*}
	  \frac{\pa u}{\pa h}(a)>0 \text{ for } a\in \Ga _D
	  \cap \pa H.
	\end{equation*}
	This completes the proof.
	%
\end{proof}
%
%
\begin{proposition}\label{grad}
	Let $\Om _s \in \A_{R_1 R_0}$ be an annular domain as given in~\eqref{family} with $0<s<R_1-R_0$. Let $u>0$ be an eigenfunction of~\eqref{eigen} corresponding to the first eigenvalue $\tau _1(s)$ in $\Om _s$. Then 
	
	\noindent
	\begin{equation*}
	\nabla u (x)\neq 0\mbox{ for }x\in \overline{\Om }_s \setminus \R e_1.
	\end{equation*}
\end{proposition}
\begin{proof}
	%
	%
	\noindent Let $x\in \overline{\Om }_s\setminus \R e_1$. Then $x_k \neq 0$ for some $k\geq 2$, where $x=(x_1,x_2,\dots ,x_d)$. From the axial symmetry of $u$, we can assume that $x_k<0$.
	\noindent Now, we consider $\eta =-x_k e_1+x_1 e_k$ and the 
	polarizer $H\dq \{x\in \RN : x\cdot \eta \leq 0 \},$
	%
	so that $\eta $ is an outward normal to $H$. Observe that $H\in \H_*$ and $x\in \overline{\Om }_s\cap \pa H$. From Proposition~\ref{propo:fss}, we get that $\frac{\pa u}{\pa \eta }(x)<0 $. Therefore $\nabla u(x)\neq 0$.
\end{proof}
Observe that, the annular domain $\Om _s=B_R(0)\setminus \overline{ B_r (s e_1)}$ is symmetric with respect to any plane containing the $e_1$-axis. Since the first eigenvalue $\tau _1(s)$ of~\eqref{eigen} is simple, the corresponding eigenfunctions inherit the same symmetry, as stated below. Also see~\cite[Proposition~A.3]{AAK}.
\begin{proposition}\label{propo:axial0}
	Let $\Om _s \in \A_{R_1 R_0}$ be an annular domain as given in~\eqref{family}. Let $u$ be an eigenfunction of~\eqref{eigen} corresponding to the first eigenvalue $\tau _1(s)$ in $\Om _s$. If $\eta \in \SN \setminus \{0\}$ such that $\eta \cdot e_1=0$, then $u$ is symmetric with respect to the hyperplane $H_\eta \dq \{x\in \RN : x\cdot \eta =0 \}$. 
	Furthermore, $\nabla u (x)\cdot \eta =0$ for $x\in \overline{\Om _s}\cap H_\eta $.
\end{proposition}
\begin{proof}
	Consider $w(x)=u(\si _\eta (x))-u(x)$ for $x\in \overline{\Om _s^-}$, where $\si _\eta $ is the reflection with respect to the hyperplane $H_\eta $. Then $w$ satisfies
	the following boundary value problem: 
	
	\noindent
	\begin{equation*}
	\left.
	\begin{aligned}
	-\De w - \tau _1(s) w =0 \mbox{ in } \Om _s^-,\\
	w = 0  \mbox{ on } (\overline{ \Om _s}\cap \pa H) \cap \Ga _D^-,\\
	\frac{\pa w}{\pa n} = 0  \mbox{ on } \Ga _N^-.
	\end{aligned}\right\}
	\end{equation*}
	From Proposition~\ref{aux:wcp}, we get that $w\geq 0$ in $\Om _s^-$. Since $-w$ also satisfies the boundary value problem above, we conclude that $w\equiv 0$ in $\overline{\Om _s^-}$. Now, by the same arguments as given in Proposition~\ref{propo:fss}, we see that $\nabla u(x)\cdot \eta = -\frac{1}{2}\nabla w(x)\cdot \eta =0$ for $x\in \overline{ \Om _s}\cap \pa H$.
\end{proof}
%
Now, we prove the monotonicity of the first eigenfunctions along the tangential directions in the following lemma (a corollary of Theorem~\ref{thm 0}) regarding. For $a\in \RN\setminus \{0\}$, let $T_a\dq \{\eta \in \RN \setminus \{0\} : \eta \cdot a=0\}$. Any $\eta \in T_a$ is called a tangential direction at $a$.
\begin{lemma}\label{coro:tan2}
	Let $\Om _s\in \A_{R_0 R_1}$ be an annualar domain as given in~\eqref{family} with $0<s<R_1-R_0$. Let $u>0$ be an eigenfunction of \eqref{eigen} corresponding to the first eigenvalue $\tau _1(s)$ in $\Om _s$. Let~$a\in \overline{\Om }_s\setminus \R e_1$ and $\eta \in T_a$. Then
	
	\noindent
	\begin{equation*}
	\nabla u(a)\cdot \eta >0 \mbox{ if and only if } \eta \cdot e_1<0.
	\end{equation*}
\end{lemma}
\begin{proof}
	First assume that $\eta \cdot e_1<0$. 
	We consider the polarizer:
	$H=\{x\in \RN : -\eta \cdot x\leq 0\}.$ Clearly $a\in \pa H$ and since $\eta \cdot e_1<0,$ we have $H\in \H _*$. Also observe that $-\eta $ is an outward direction to
	$H$ at $a\in \pa H$.
	So, from Proposition~\ref{propo:fss}, we get that
	$\nabla u(a)\cdot \eta >0.$

	\noindent Conversely, assume that $\nabla u(a)\cdot \eta >0$. If possible, let $\eta \cdot e_1 \geq 0$. 
	If (i) $-\eta \cdot e_1 <0$, then the arguments given above implies $\nabla u(a)\cdot (-\eta )> 0$, or (ii) $-\eta \cdot e_1 =0$, Proposition~\ref{propo:axial0} gives $\nabla u(a)\cdot (-\eta )=0$. In both the cases we get contradictions to the assumption that $\nabla u(a)\cdot \eta >0$. Thus $\eta \cdot e_1 <0$. 
\end{proof}
\begin{remark}\label{rmk:fSS2}
	From the characterization in Remark~\ref{rmk:fSS1}, Lemma~\ref{coro:tan2} is equivalent to $\frac{\pa u}{\pa \T}(r,\T)<0$ for $0<r\leq R_1$ and for all of the possible values of $\T$ in $\Om_s$, when we consider $u=u(r,\T)$ as a function of $r$ and $\T$.
\end{remark}
\begin{proposition}\label{propo:axial}
	Let $\Om _s\in \A_{R_0 R_1}$ be an annualar domain as given in~\eqref{family} with $0<s<R_1-R_0$. Let $u>0$ be an eigenfunction of \eqref{eigen} corresponding to the first eigenvalue $\tau _1(s)$ in $\Om _s$.
	Then 
	
	\noindent
	\begin{equation*}
	\frac{\pa u}{\pa x_1}<0 \mbox{ on } \Ga _N\setminus \{\pm R_1 e_1\}.
	\end{equation*}
\end{proposition}
\begin{proof}
	Let $a\in \Ga _N\setminus \{\pm R_1 e_1\}$. From the Neumann condition on $\Ga _N$, we get $\nabla u(a)\in T_a$. Since $\nabla u$ never vanishes on $\Ga _N\setminus \{\pm R_1 e_1\}$, we obtain $\nabla u(a)\cdot \nabla u(a)>0$. Therefore, from Lemma~\ref{coro:tan2}, we get
	
	\noindent
	\begin{equation*}
	\frac{\pa u}{\pa x_1}(a)=\nabla u(a)\cdot e_1<0. \qedhere
	\end{equation*}
\end{proof}
\begin{remark}
	One can give an alternate proof for Proposition~\ref{propo:axial}. The axial symmetry of $u$ and the Neumann condition $\frac{\pa u}{\pa n}=0$ on $\Ga _N$ imply that 
	\begin{equation}\label{eqn:gradinpolarform}
	\nabla u (x)=\frac{\pa u}{\pa \T}\left(\frac{\sin \T}{R_1}e_1-\frac{x_1(x-(x\cdot e_1)e_1)}{R_1^2 |x-(x\cdot e_1)e_1|}\right) \mbox{ for }x\in \Ga _N\setminus\{\pm R_1 e_1\}.
	\end{equation}
	Therefore, 
	
	\noindent
	\begin{equation*}
	\frac{\pa u}{\pa x_1}=\frac{1}{R_1}\frac{\pa u}{\pa \T}\sin \T<0
	\end{equation*}
	as $\frac{\pa u}{\pa \T}<0$ from the foliated Schwarz symmetry (see Remark~\ref{rmk:fSS2}). The equation~\eqref{eqn:gradinpolarform} is derived by considering the coordinate system $(r,\T, \varphi _1,\dots,\varphi _{d-2})$ with $r=|x|,r\cos \T=-x_1$ and $\T\in [0,\pi),\varphi _i\in [0,2\pi )$ for $i=2,\dots ,d-2,$ as given in \cite[Proof of Lemma~5.1 and Proposition~5.4]{Pedro-Tobis}, and observing that the axial symmetry implies that $\frac{\pa u}{\pa \varphi _i}=0$ for $i=1,2,\dots , d-2$ and the Neumann condition on $\Ga _N$ implies that $\frac{\pa u}{\pa r}=0$ on $\Ga _N$ (i.e., for $r=R_1$).
\end{remark}
\subsection{Strict monotonicity along the affine-radial directions}\label{Afine-4.2}
For a point $x\in \RN \setminus \{0\}$, the vector $x-a$ is called an \emph{affine-radial direction} from $a\in \RN $. The following theorem gives the monotonicity of the first eigenfunctions of~\eqref{eigen} along the affine-radial directions from $s e_1$.
\begin{theorem}\label{thm 3}
	Let $\Om _s\in \A_{R_0 R_1}$ be an annualar domain as given in~\eqref{family} with $0<s<R_1-R_0$. Let $u>0$ be an eigenfunction of \eqref{eigen} corresponding to the first eigenvalue $\tau _1(s)$ in $\Om _s$. Then $u$ is strictly increasing along all the affine-radial directions from $se_1$ in $\Om _s$, more precisely
	
	\noindent
	\begin{equation*}
	\nabla u(x)\cdot (x-se_1)>0 \mbox{ for }x\in \overline{\Om }_s\setminus \{\pm R_1 e_1\}.
	\end{equation*}
\end{theorem}
\begin{proof}
	We consider the function
	$ v(x)= \nabla u(x)\cdot (x-se_1)\mbox{ for } x\in \overline{\Om }_s.$
	Then 
	%
	$v$ satisfies the following equation:
	
	\noindent
	\begin{equation*}
	-\De v - \tau _1(s) v=2\tau _1(s) u \mbox{ in }\Om _s.
	\end{equation*} 
	Notice that, for $a\in \Ga _N$, $a$ is an outward normal to $\Om _s$ therefore $ \nabla u (a)\cdot a =0$. Thus, 
	
	\noindent
	\begin{equation*}
	v(a)=\nabla u(a)\cdot a -s\frac{\pa u}{\pa x_1}(a)
	=-s\frac{\pa u}{\pa x_1}(a). 
	\end{equation*}
	Now, by Proposition~\ref{propo:axial}, we have $v(a)>0$ for $a\in \Ga _N\setminus \{\pm R_1 e_1\}$. Since $e_1$ is a normal direction at $a=\{\pm R_1 e_1\}$ thus $v(a)=0$.
	For~$a\in \Ga _D$, $a-se_1$ is an inward direction to $\Om _s$ at $a$ and $u(a)=0$. Therefore, from Hopf's lemma (Proposition~\ref{SMP}) we get 
	$v(x)=\nabla u(x)\cdot (x-se_1)>0.$
	So $v$ satisfies the following boundary value problem: 
	
	\noindent
	\begin{equation*}
	\left.
	\begin{aligned}
	-\De v - \tau _1(s) v =& \ 2\tau _1(s) u\geq 0 \mbox{ in } \Om _s,\\
	v > 0 \mbox{ on } \Ga _D\cup \big(\Ga _N\setminus \{\pm R_1 e_1\}\big),& 
	\quad v =0 \mbox{ on }\{\pm R_1 e_1\}.
	\end{aligned}
	\right\}
	\end{equation*}
	Since $\tau _1(s)<\la _1(s)$, from Proposition~\ref{aux:smp} we get $v\geq 0$ in $\Om _s$. Now, using the strong maximum principle (Proposition~\ref{SMP}) we conclude that $v>0$ on $\overline{\Om }_s\setminus \{\pm R_1 e_1\}$.
\end{proof}
\begin{remark}
	\noi
	\begin{enumerate}[\rm (i)]
		\item
		For the concentric annular domain $\Om _0$, i.e., for $s=0$, Theorem~\ref{thm 3} asserts that a positive first eigenfunction $u$ of~\eqref{eigen} is strictly increasing along the radial direction, i.e.,
		$\frac{\pa u}{\pa r}>0$ in $\Om _0$.
		\item
		Since $u$ is strictly increasing along all the affine-radial directions from $s e_1$, we remark that $\nabla u$ can never vanishes in $\overline{ \Om }_{s}\setminus \{\pm R_1 e_1\}$, as concluded in Proposition~\ref{grad}.
		\item 
		We observe that the foliated Schwarz symmetry (Theorem~\ref{thm 0}) of the positive first eigenfunctions implies that the maximum is attained only on the $e_1$-axis. Then, the monotonicity in the affine-radial directions implies that the maximum is attained at $-R_1 e_1$. 
	\end{enumerate}
\end{remark}
%
\subsection{Strict monotonicity along the axial direction}\label{Axial-4.3}
In this section, we prove the monotonicity of the positive eigenfunctions of~\eqref{eigen} corresponding to $\tau_1(s)$ along the $e_1$-direction in a subdomain of $\Om _s$. First, we introduce some notations. For $\al \in (-R_1,R_1)$ the $\al $-cap of $\Om _s$, denoted by $\Sigma _\al $, is defined as:

\noindent
\begin{equation*}
\Si _\al =\{x\in \Om _s : x_1<\al \}.
\end{equation*}
Let $H_\al = \{x\in \RN : x_1\leq \al \}$ be a polarizer,
$\Ga _D^\al =\Ga _D\cap H_\al $ and
$\Ga _N^\al =\Ga _N\cap H_\al $. Therefore
$\pa \Si _\al =\Ga _N^\al \sqcup \Ga _D^\al \sqcup (\overline{\Om _s}\cap \pa H_\al).$  See Figure~\ref{fig:cap}.
\begin{figure}[tb]
	\includegraphics[width=0.3\linewidth]{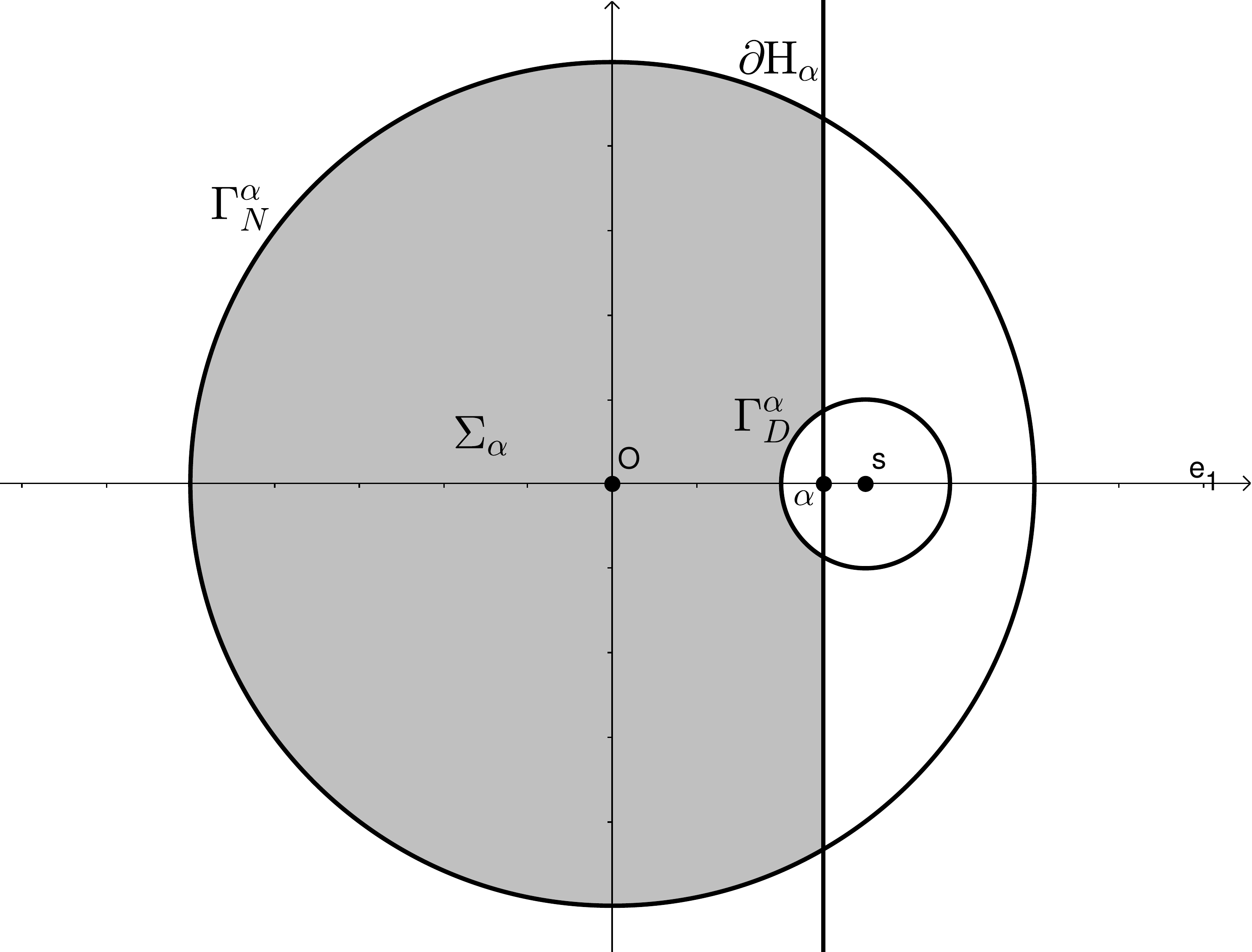}
	\caption{The $\al$-caps of the annular domain $\Om _s$.}
	\label{fig:cap}
\end{figure}
We observe that
\begin{enumerate}[\rm (i)]
	\item $\Ga _D^\al \neq \emptyset \mbox{ only if }\al >s-R_0,$
	\item $\Si _{\al _1} \subset \Si _{\al _2} ,\ \Ga _D^{\al _1} \subset \Ga _D^{\al _2}\mbox{ and } \Ga _N^{\al _1} \subset \Ga _N^{\al _2} \mbox{ for }\al _1 <\al _2 .$
\end{enumerate}
Let $\si _\al $ be the reflection with respect to $H_\al $. Then, we have

\noindent
\begin{equation*}
\si _\al (x)=(2\al -x_1,x')=2\al e_1+\si _0(x)\quad \forall x =(x_1,x')\in \RN.
\end{equation*}
%
%
%
%
\begin{lemma}\label{le:s/2}
	Let $\Om _s\in \A_{R_0 R_1}$ be an annualar domain as given in~\eqref{family} with $0<s<R_1-R_0$. Let $u>0$ be an eigenfunction of \eqref{eigen} corresponding to the first eigenvalue $\tau _1(s)$ in $\Om _s$. If $ \al \leq 0$ then 
	
	\noindent
	\begin{equation*}
	\frac{\pa u}{\pa x_1} (x)<0\mbox{ for }x\in \overline{\Si }_\al \setminus \{-R_1 e_1\}.
	\end{equation*}
\end{lemma}
\begin{proof}
	Since $\Si _\al \subseteq \Si _{ 0} \mbox{ for }\al \leq 0,$ it is enough to prove the result for
	$\al =0$. Assume that $\al =0$. First, we observe that $\frac{\pa u}{\pa x_1}$ satisfies the following equation:
	
	\noindent
	\begin{equation*}
	-\De\left(\frac{\pa u}{\pa x_1}\right)-\tau _1(s)\frac{\pa u}{\pa x_1}=0 \mbox{ in }\Om _s.
	\end{equation*}
	\noindent Notice that, $e_1$ is an outward direction to $\Om _{s}$ for $x\in \Ga _D^0$ (if non-empty). We have $\frac{\pa u}{\pa x_1}<0$ on $\Ga _D^0$ (Proposition~\ref{SMP}) and on $\Ga _N^0\setminus \{-R_1 e_1\}$ (Proposition~\ref{propo:axial}). For $x\in \pa H_0 \cap \overline{\Om }_s$ we have $x\cdot e_1=0$, i.e., $e_1\in T_x$ therefore Lemma~\ref{coro:tan2} implies that $\frac{\pa u}{\pa x_1}(x)<0 
	$.
	Therefore,~$\frac{\pa u}{\pa x_1} $ satisfies the following boundary value problem in
	$\Si _0 :$ 
	
	\noindent
	\begin{equation*}
	\begin{array}{c}
	-\De\left(\frac{\pa u}{\pa x_1}\right)-\tau _1(s)\frac{\pa u}{\pa x_1} =0\mbox{ in } \Si _0 ,\\
	\frac{\pa u}{\pa x_1}
	\leq 0 \mbox{ on }\Ga _N^0, \
	\frac{\pa u}{\pa x_1}
	<0 \mbox{ on } \Ga _D^0\cup (\pa H_0 \cap \overline{\Om }_s).
	\end{array}
	\end{equation*}
	Thus, using Proposition~\ref{aux:smp} and Proposition~\ref{SMP} 
	we obtain that $\frac{\pa u}{\pa x_1}<0 \mbox{ on } \overline{\Si }_0 \setminus \{-R_1 e_1\}$.
\end{proof}
\noindent Now, let us consider the following set:

\noindent
\begin{equation*}
\A \dq \left\{\al \in (-R_1,s) : \frac{\pa u}{\pa x_1}<0 \mbox{ in } \overline{\Si }_\al \setminus \{-R_1 e_1\} \right\}. 
\end{equation*}
Observe that $\A \neq \emptyset $, since $(-R_1,0)\subset \A $ (from Lemma~\ref{le:s/2}). 
We claim that the maximal cap such that $\displaystyle \frac{\pa u}{\pa x_1}<0$ is the $s$-cap $\Si _s$.
\begin{theorem}\label{thm 2}
	Let $\Om _s\in \A_{R_0 R_1}$ be an annular domain as given in~\eqref{family} with $0<s<R_1-R_0$. Let $u>0$ be an eigenfunction of \eqref{eigen} corresponding to the first eigenvalue $\tau _1(s)$ in $\Om _s$. Then
	
	\noindent
	\begin{equation*}
	\displaystyle \frac{\pa u}{\pa x_1}(x)<0\mbox{ for } x\in \Si _s.
	\end{equation*}
\end{theorem}
\begin{proof}
	We observe that $\frac{\pa u}{\pa x_1}(x)<0$ for $x\in \Ga _N\setminus \{\pm R_1 e_1 \}$ (from Proposition \ref{propo:axial}), also for $x\in \Ga _D\cap \{x_1<s\}$ (from Hopf's lemma). Now, let $\al \in \A$ with $\al <s$. Since $\frac{\pa u}{\pa x_1}(x)<0$ for $x\in \big(\Ga _D\cap \{x_1<s\} \big)\cup \big(\Ga _N\setminus \{\pm R_1 e_1\} \big)$, by the continuity of $\frac{\pa u}{\pa x_1}$ and the compactness $\pa H_\al \cap \overline{\Om }_s$ there exists $\epsilon (\al )>0$ such that $\frac{\pa u}{\pa x_1}(x)<0$ 
	for $x\in \Si _{\al +\epsilon (\al)}$, i.e., $\al +\epsilon (\al )\in \A.$ Also we observe, again from the Hopf's lemma, that $\frac{\pa u}{\pa x_1}(x)\geq 0$ for $x\in \Ga _D\cap \{x_1\geq s \}$. Hence $\sup \A =s$. Therefore $\frac{\pa u}{\pa x_1}(x)\leq 0$ for $x\in \overline{\Si }_s$. Since $\frac{\pa u}{\pa x_1}$ satisfies the following equation in ${\Si }_s:$ 
	
	\noindent
	\begin{equation*}
	-\De \left(\frac{\pa u}{\pa x_1}\right)= \tau _1(s) \frac{\pa u}{\pa x_1} \leq 0 ,
	\end{equation*}
	and 
	$\frac{\pa u}{\pa x_1}<0$ in $\Si _{0}$, from the strong maximum principle (Proposition~\ref{SMP}) we obtain that $\frac{\pa u} {\pa x_1}<0$ in $\Si _s.$
\end{proof}
\section{The shape variation of the first eigenvalue}\label{Mono}
In this section, we apply the geometry of the eigenfunctions that we obtained in Section~\ref{Symm and Mono} to study the monotonicity of the first eigenvalue of~\eqref{eigen} over a family annular domains. 
Recall that,
$\Om_s=B_{R_1}(0)\setminus \overline{B_{R_0}(se_1)}$
is an annular domain for some
$0<R_0<R_1\mbox{ and }0\leq s<R_1-R_0$; and $\tau _1(s)$ is the first eigenvalue of \eqref{eigen} in $\Om _{s}$. For a given vector field $V\in W^{3,\infty }(\RN , \RN )$, we consider the perturbations of $\Om _s$ by $\widetilde{\Om }_\eps =(I+\eps V)\Om _s$. We consider the vector field $V$ as given below:

\noindent
\begin{equation*}
	V(x)=\rho (x) e_1, \mbox{ where } \rho \in \C_c^{\infty }(B_{R_1}(0))\mbox{ and } \rho \equiv 1 \mbox{ in a neighborhood of } B_{R_0}(s e_1).
\end{equation*}
For this choice of $V$ and $\eps $ sufficiently small, the perturbations $\widetilde{\Om }_\eps $ of $\Om _s$ are precisely the translation of the inner ball along the $e_1$-direction towards the outer boundary. More precisely, 

\noindent
\begin{equation*}
	\widetilde{\Om }_\eps =\Om _{s+\eps }.
\end{equation*}
Therefore,
$\tau _1(\widetilde{\Om }_\eps )=\tau _1(s+\eps ),\ \tau _1 (\widetilde{\Om }_0)=\tau _1(s)$. Let $u>0$ be the eigenfunction corresponding to $\tau _1(s)$ with the normalization $\int _{\Om_{s}}u^2\dx=1$.
Hence, from the definition of the Eulerian derivative of the shape functionals and the shape derivative formula~\eqref{shape formula}, we get the following expression for $\tau '_1(s):$
\begin{equation}\label{shape}
	\tau _1'(s)=-\int _{\Ga_D}\left|\frac{\pa u}{\pa n}(x)\right|^2 n_1(x) \dS .
\end{equation}
%
Recall, from Section~\ref{Axial-4.3},  that $\Si _s$ is the $s$-cap of $\Om _{s}$ and $\si _s$ is the reflection with respect to the boundary $\pa H_s$ of the polarizer
$H_s= \{x\in \RN : x_1\leq s\}.$ Therefore,
\begin{equation}\label{reflection}
	\si_s(x)=x(I-2{e_1}^T e_1)+2s e_1=x\si_0+2s e_1\quad \forall x\in \RN, 
\end{equation}
where $\si_0$ is the symmetric matrix $I-2{e_1}^T e_1.$
%
Now, we rewrite the derivative formula \eqref{shape} as an integral over the half boundary $\Ga _D\cap \{x_1>s \}$.
\begin{proposition}\label{rewrite}
	Let $\Om _s\in \A_{R_0 R_1}$ be an annualar domain as given in~\eqref{family}. Let $u>0$ be an eigenfunction of \eqref{eigen} corresponding to the first eigenvalue $\tau _1(s)$ in $\Om _s$ with the normalization $\int_{\Om_{s}}u^2 \dx =1$. Then
	\begin{equation}\label{shape1}
	\tau_1'(s)=\Int_{\Ga _D\cap \{x_1>s \} } \left(\left|\frac{\pa v}{\pa n}(x)\right|^2 -\left|\frac{\pa u}{\pa n}(x)\right|^2\right)n_1(x)\dS,
	\end{equation}
	where $v=u\circ \si _s$ in $\overline{\Om _s}\cap \{x_1 \geq s \}$.
\end{proposition}
\begin{proof}
	Firstly, we split the right hand side of the equation~\eqref{shape} as follows:
	\begin{equation}
	\tau _1'(s)=-\Int _{\Ga _D\cap \{x_1>s \} } \left|\frac{\pa u}{\pa n}(x)\right|^2 n_1(x) \dS-\Int _{\Ga _D\cap \{x_1< s \}} \left|\frac{\pa u}{\pa n}(x)\right|^2 n_1(x) \dS.
	\end{equation}
	Since $\si _s(\Om _s\cap H_s\cm )\subset \Om _{s}$, we define
	$v(x)=u(\si _s(x))\mbox{ for } x\in \Om _s\cap H_s\cm .$
	Then, we have
	\begin{equation}\label{eqn:grad-reflection}
		\nabla v(x)=\nabla u(\sigma _s(x))\si_0 \quad \forall x\in \Om _s\cap H_s\cm .
	\end{equation}
	Observe that $n(\si _s(x))=n(x)\si _0$ and $n_1(\si _s(x))=-n_1(x)$ for $x\in \Ga _D\cap \{x_1>s \}$. Therefore, from~\eqref{eqn:grad-reflection} we obtain
	
	\noindent
	\begin{equation*}
		\frac{\pa v}{\pa n}(x)=\nabla v(x)\cdot n(x)=\nabla u(\si _s (x))\si _0\cdot n(x)=\nabla u(\si _s(x))\cdot n(x)\si _0 =\frac{\pa u}{\pa n}(\si _s(x)).
	\end{equation*}
	Since $\si _s(\Ga _D)=\Ga _D$, the second integral can be written as
	\begin{eqnarray*}
		\Int _{\Ga _D\cap \{{x_1< s}\}}\left|\frac{\pa u}{\pa n}(x)\right|^2 n_1(x) \dS
		&=&\Int _{\Ga _D\cap \{{x_1> s}\}} \left|\frac{\pa u}{\pa n}(\si _s(x))\right|^2 n_1(\si _s(x)) \dS 
		=-\Int _{\Ga _D\cap \{{x_1> s}\}} \left|\frac{\pa v}{\pa n}(x)\right|^2 n_1(x) \dS.
	\end{eqnarray*}
	This gives~\eqref{shape1}.
%
\end{proof}
Now, we have the following lemma regarding the Neumann data of $u$ on the reflected part in $\Om _s$ of the outer boundary $\Ga _N$.
\begin{lemma}\label{lem:Neumann} 
	Let $\Om _s\in \A_{R_0 R_1}$ be an annualar domain as given in~\eqref{family}. Let $u>0$ be an eigenfunction of \eqref{eigen} corresponding to the first eigenvalue $\tau _1(s)$ in $\Om _s$. 
	Then
	
	\noindent
	\begin{equation*}
		\frac{\pa v}{\pa n}(x)>0 \mbox{ for }x\in \Ga _N \cap \{x_1>s \}.
	\end{equation*}
\end{lemma}
\begin{proof}
	%
	Observe that $\nabla v(x)=\nabla u(\si _s(x))\si _0$ for $x\in \overline{ \Om _s}$, and the unit normal to $\Om _s$ at $x\in \Ga _N$ is $n(x)=\frac{x}{R_1}$. Let $x\in \Ga _N\cap \{x_1>s \}$. Using \eqref{reflection}, we get that
	\begin{eqnarray*}
		\frac{\pa v}{\pa n}(x)&=&\nabla u(\si _s(x))\si _0 \cdot n(x)=
		\frac{1}{R_1}\nabla u(\si _s(x))\cdot \si _0(x)\\ &=& \frac{1}{R_1}\nabla u(\si _s(x))\cdot (\si _s (x)-2se_1) = \frac{1}{R_1}\nabla u(\si _s(x))\cdot (\si _s (x)-se_1)-\frac{s}{R_1}\frac{\pa u}{\pa x_1}(\si _s (x)).
	\end{eqnarray*}
    For $x\in \Ga _N\cap \{x_1>s \} $, we have $\si _s(x)\in \Om _s$. Therefore, from Theorem~\ref{thm 3} and Theorem~ \ref{thm 2}, we have
    
    \noindent
    \begin{equation*}
    	\nabla u(\si _s(x))\cdot (\si _s (x)-se_1)>0\text{ and }\frac{\pa u}{\pa x_1}(\si _s(x))<0.
    \end{equation*}
    Hence, \[\frac{\pa v}{\pa n} (x)>0\mbox{ for }x\in \Ga _N \cap \{x_1>s \}.\qedhere \]
\end{proof}
Now, we give a proof of Theorem~\ref{thm}. 
%
\begin{proof}[Proof of Theorem~\ref{thm}]
	For $s=0$, $\tau '_1(0)=0$ follows easily from the radial symmetry of the first eigenfunctions. Let $0<s<R_1-R_0$. We consider the  function
	$w(x)\dq v(x)-u(x) \mbox{ for } x\in \overline{\Om _{s} \cap H_s\cm }.$ Clearly, $w$ satisfies the following equation:
	
	\noindent
	\begin{equation*}
		-\De w -\tau _1(s)w=0\mbox{ in } {\Om _{s}\cap H_s\cm }.
	\end{equation*}
	Since $\si _s(\Ga _D)=\Ga _D$ we get that $v(x)=0$ on $\Ga _D\cap \{x_1\geq s \}$. Therefore, from the definition of $w$, it follows that $w=0$ on $(\Ga _D \cap \{x_1\geq s \}) \cup (\Om _{s}\cap \pa H_s)$. From Lemma~\ref{lem:Neumann}, we have
	$\frac{\pa w}{\pa n} (x)=\frac{\pa v}{\pa n}(x)-\frac{\pa u}{\pa n} (x) =\frac{\pa v}{\pa n}(x)>0$  for $x\in \Ga _N\cap \{x_1>s \}$.
	Therefore, $w$ satisfies the following boundary value problem:
	
	\noindent
	\begin{equation*}
		\left.
		\begin{aligned}
			-\De w-\tau _1(s) w&=0 \mbox{ in }\Om _{s}\cap H_s\cm,\\
			w&=0 \mbox{ on } (\Ga_D\cap \{x_1>s \})\cup (\Om _{s}\cap \pa H_s),\\
			\frac{\pa w}{\pa n}&>0\mbox{ on }\Ga_N\cap \{x_1>s \}.
		\end{aligned}
	    \right\}
    \end{equation*}
    From Proposition~\ref{aux:wcp} and the strong maximum principle (part (a) of Proposition~\ref{SMP}), we obtain that $w>0$ in $\Om _s\cap H_s\cm$. Since $w=0$ and $u=0$ on $\Ga_D\cap \{x_1>s \}$, by Hopf's lemma (Proposition~\ref{SMP}), we get $\frac{\pa w}{\pa n}<0$ and $\frac{\pa u}{\pa n}<0$ on $\Ga_D\cap \{x_1>s \}.$ 
    Therefore, $\frac{\pa v}{\pa n}(x)<\frac{\pa u}{\pa n}(x)<0$ for $x\in \Ga _D\cap \{x_1>s \}$. Hence 
    \[\left|\frac{\pa v}{\pa n}(x)\right|^2 -\left|\frac{\pa u}{\pa n}(x)\right|^2>0 \text{ for } x\in \Ga _D\cap \{x_1>s \}.\]
    Consequently, since $n_1(x)<0$ for $x\in \Ga _D\cap \{x_1>s \}$, from \eqref{shape1} we conclude that $\tau _1'(s)<0.$
\end{proof}
	\section{Remarks on other boundary value problems}\label{remarks}
For an annular domain $\Om _s =B_{R_1}(0)\setminus \overline{ B_{R_0} (s e_1)}\in \A_{R_0, R_1}$, 
we consider the following eigenvalue problems:

\medskip
\begin{minipage}{0.45\textwidth}
	\begin{equation}\label{DN}\tag{D-N}
		\left.
		\begin{aligned} 
			-\De u & = \nu u \mbox{ in } \Om _s,\\
			u &= 0 \mbox{ on } \pa B_{R_1}(0),\\
			\frac{\pa u}{\pa n} &= 0 \mbox{ on } \pa B_{R_0}(s e_1);
		\end{aligned}
		\right\}
	\end{equation}
\end{minipage}
\begin{minipage}{0.45\textwidth}
	\begin{equation}\label{DD}\tag{D-D}
		\left.
		\begin{aligned} 
			-\De u & = \la u \mbox{ in } \Om _s,\\
			u &= 0 \mbox{ on } \pa \Om _s.
		\end{aligned}
		\right\}
	\end{equation}
\end{minipage}
\medskip

\noindent Let $\nu _1(s),\ \la _1(s)$ be the first eigenvalues of \eqref{DN} and \eqref{DD} respectively.
By the similar arguments used in proving Theorem~\ref{thm 0}, we can prove the following results. We omit the proofs. 
\begin{theorem}
	Let $\Om _s \in \A_{R_0, R_1}$ be an annular domain as given in~\eqref{family} with $0<s<R_1-R_0$. Let $u>0$ be an eigenfunction of~\eqref{DN} in $\Om _s$ corresponding to the first eigenvalue $\nu _1(s)$.
	Then $u$ has the foliated Schwarz symmetry in $\Om _s$ with respect to the 1-dimensional closed affine half subspace $s e_1 - \R^+ e_1$.
\end{theorem}
\begin{theorem}
	Let $\Om _s \in\A_{R_0,R_1}$ be an annular domain as given in~\eqref{family} with $0<s<R_1-R_0$. Let $u>0$ be an eigenfunction of~\eqref{DD} in $\Om _s$ corresponding to the first eigenvalue $\la _1(s)$.
	Then,
	\begin{enumerate}[$(\rm a)$]
		\item $u$ has the foliated Schwarz symmetry in $\Om _s$ with respect to the 1-dimensional closed affine half subspace $-\R^+ e_1$;
		\item $u$ has the foliated Schwarz symmetry in $\Om _s$ with respect to the 1-dimensional closed affine half subspace $s e_1 - \R^+ e_1$.
	\end{enumerate}
\end{theorem}
\noindent
\textbf{Open Problems:} The monotonicity of the first eigenvalue $\la _1(s)$ of~\eqref{DD} on $[0,R_1-R_0)$ is proved in~\cite{Kesavan,Harrell}, see Introduction for more details on this problem.  Notice that, for~\eqref{DN}, the first eigenfunctions have the foliated Schwarz symmetry, however the other symmetries, such as the affine radiality and the monotonicity in the axial direction clearly fail, see Figure~\ref{fig:conj4}. Since these geometries play the important roles in our proof for the monotonicity of $\tau _1(s)$ for \eqref{eigen} problem, we anticipate that $s\longmapsto \nu _1(s)$ is not monotone on $[0,R_1-R_0)$. Indeed, there is numerical evidence to support our intuition (see Figures~\ref{fig:conj1}-\ref{fig:conj3}). From the numerical data, we observe that 
\begin{observation}\label{obs1}
	There exists a unique $s_0\in(0,R_1-R_0]$ such that \[\nu_1'(s)<0\mbox{ on } (0,s_0),\ \nu_1'(s)>0\mbox{ on } (s_0,R_1-R_0) \mbox{ and }\nu_1'(s_0)=0. \]
\end{observation}
\begin{observation}\label{obs2}
	There exists a critical value, say $r_0$, of the ratio $\frac{R_0}{R_1}$ such that:
	\begin{enumerate}[(a)]
		\item For $\frac{R_0}{R_1}<r_0$, Observation~\ref{obs1} holds;
		\item For $\frac{R_0}{R_1}\geq r_0$, we have
		$\nu _1'(s)<0 \mbox{ for } s\in (0,R_1-R_0).$
	\end{enumerate}
\end{observation}
Providing analytical explanations for the observations above seems to be a challenging problem.
\begin{figure}[tb]\label{fig2}
	\centering
	\begin{subfigure}{0.49\linewidth}
		\centering
		\includegraphics[width=0.5\linewidth]{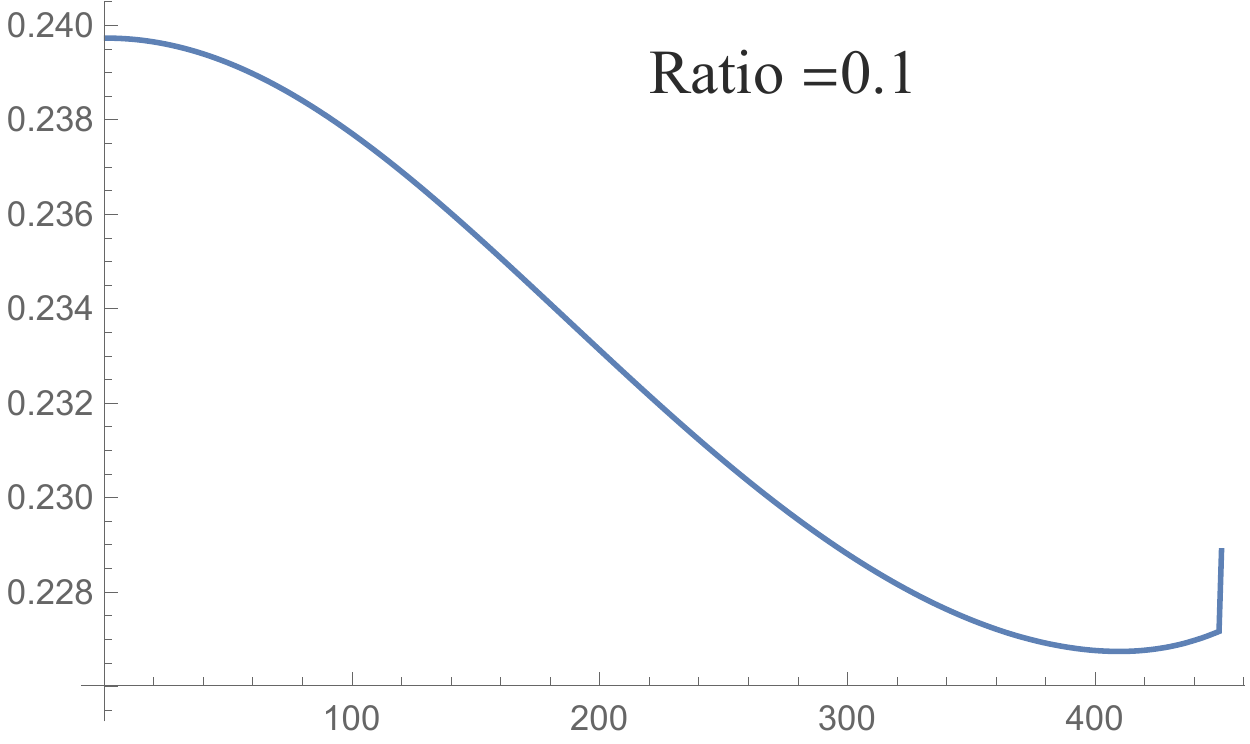}
		\subcaption{}
		\label{fig:conj1}
	\end{subfigure}
	\begin{subfigure}{0.49\linewidth}
		\centering
		\includegraphics[width=0.5\linewidth]{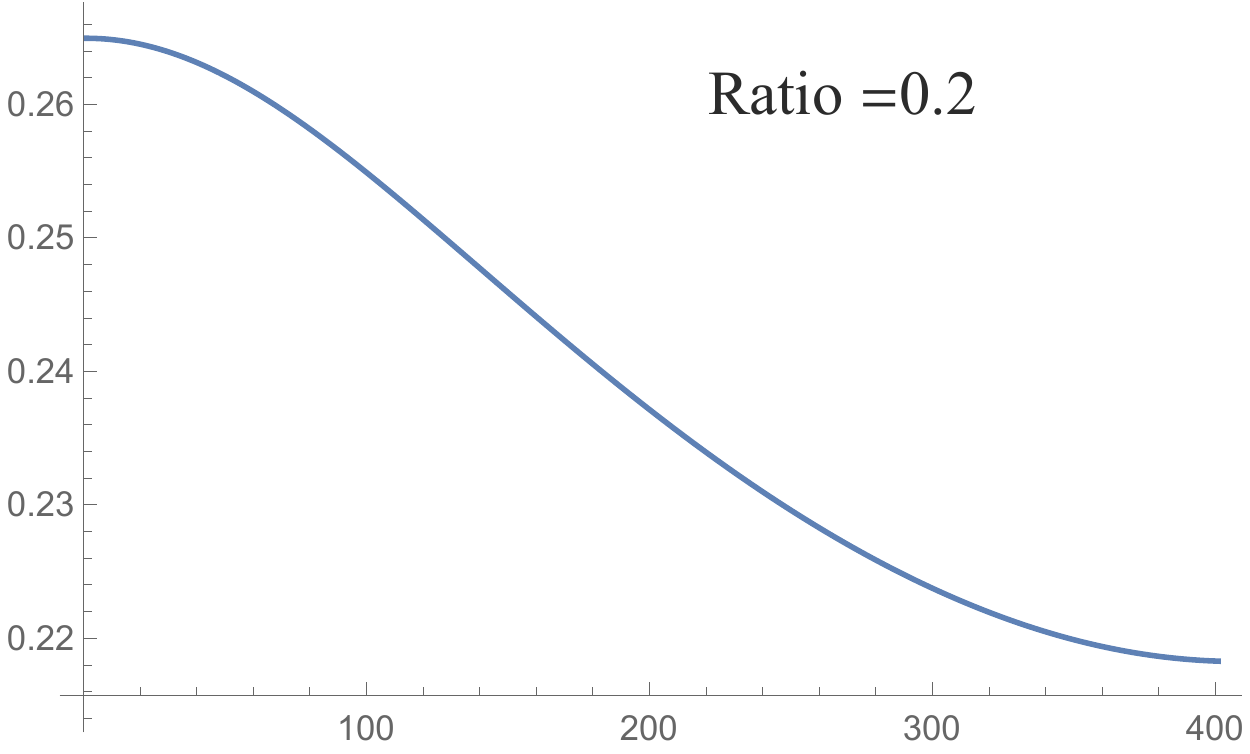}
		\subcaption{}
		\label{fig:conj2}
	\end{subfigure}
	\begin{subfigure}{0.49\linewidth}
		\centering
		\includegraphics[width=0.5\linewidth]{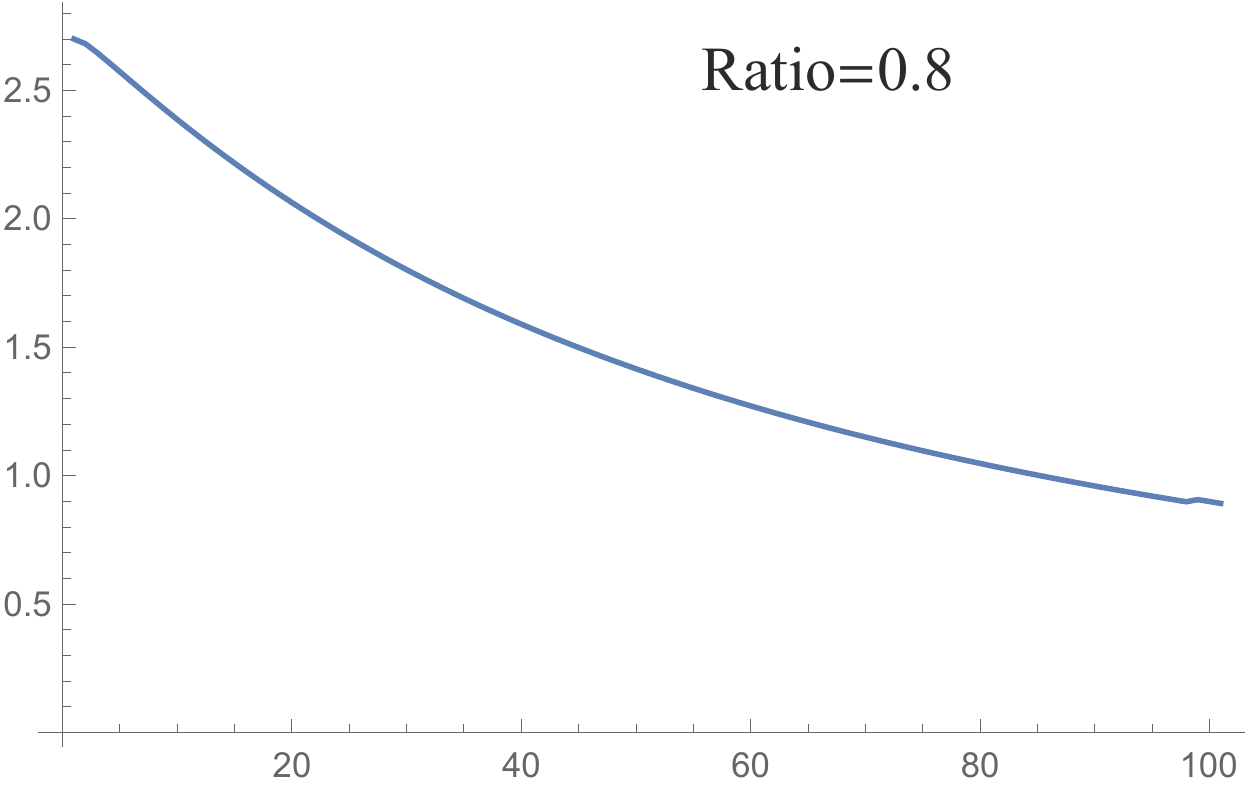}
		\subcaption{}
		\label{fig:conj3}
	\end{subfigure}
	\begin{subfigure}{0.49\linewidth}
		\centering
		\includegraphics[width=0.5\linewidth]{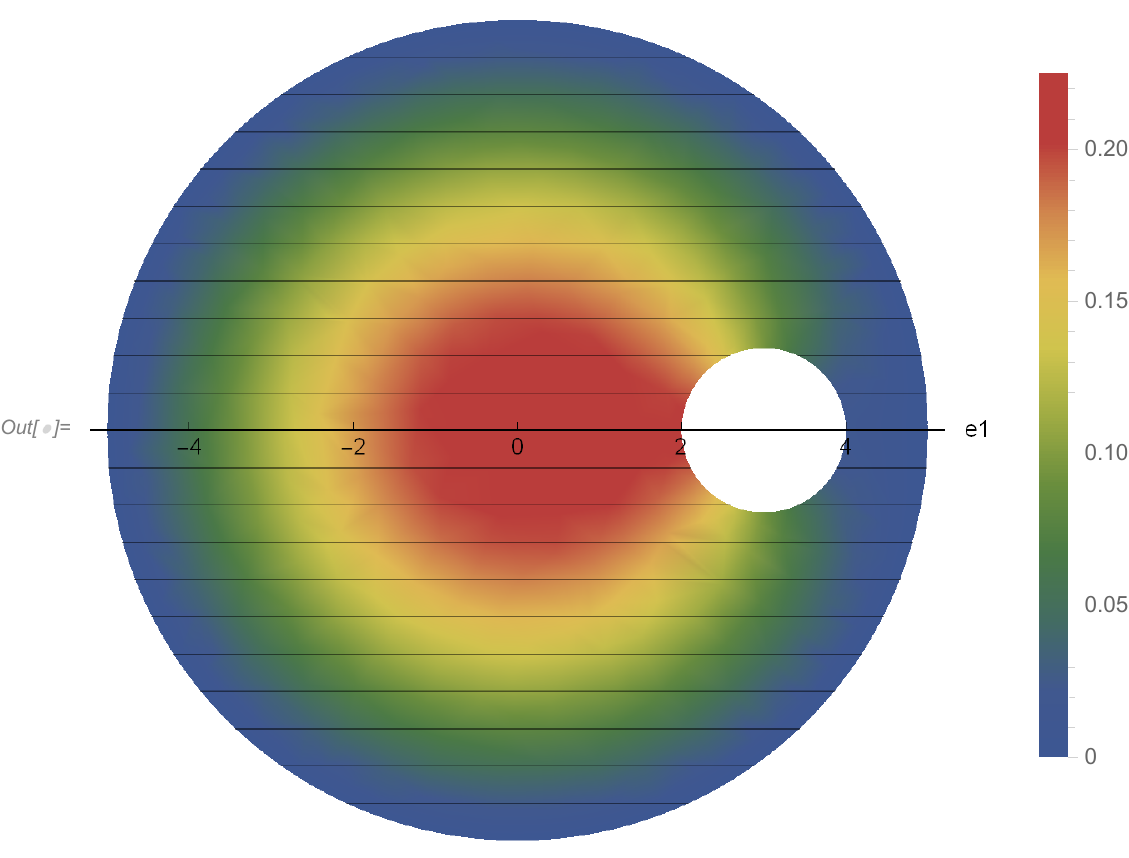}
		\subcaption{}
		\label{fig:conj4}
	\end{subfigure}
	\caption{(A)-(C) Monotonicity of $\nu _1(s)$ on $(0, R_1-R_0)$ with $R_1=5$ and $R_0=\mbox{Ratio} \times R_1$. (D) Monotonicity along the axial direction of an eigenfunction corresponding to $\nu _1(s)$.}
\end{figure}
\appendix
\section{Shape calculus for the first eigenvalue}\label{Shape}
In this section, we derive the shape derivative formula for the first eigenvalue of the Laplace operator with the mixed boundary conditions on a bounded domain. First, we recall the concept of a shape (domain) functional on an admissible family of subsets of $\RN$, and we study its variations under certain domain perturbations.
\begin{definition}[Shape functional]
	Let $\emptyset \neq D\subseteq \RN$ be an open set and let $\P(D)$ be the power set of $D$. A shape functional $J$ is a map from an admissible family $\A \subseteq \P(D)$ into $\R$.
\end{definition}
To study the variations of a shape functional $J$ at a domain $\Om \in \A$, first we consider a family of perturbations of $\Om \subset \RN$ given as follows: 
%
\begin{equation}\label{trans}
    \left.
	\begin{array}{c}
	\text{for a fixed vector field } V\in W^{3,\infty }(\RN ,\RN) \text{ we consider the transformations}\\ \Phi _\eps (x)=(I+\eps V)x=x+\eps V(x) \quad \mbox{for } x\in \RN ,\ \eps \in [0,\eps _0) \mbox{ for some small } \eps _0 >0;
	\end{array}
	\right\}
\end{equation}
%
and a family of perturbed domains
\begin{equation}\label{admisfamily}
\A \dq \Big\{\Om _\eps =\Phi _\eps (\Om ): \eps \in [0,\eps _0 )\Big\}\text{
	for some small }\eps _0 >0.
\end{equation}
\begin{definition}[Eulerian derivative]\label{Def:shape}
	Let	$V$ be a vector field as given in~\eqref{trans}, and $J$ be a shape functional. The Eulerian derivative of the shape functional $J$ at $\Om \in \A$ in the direction of $V$ is defined as the following limit, if it exists:
	
	\noindent
	\begin{equation*}
		dJ(\Om ;V)=\lim\limits_{\eps \searrow 0} \frac{J(\Om _\eps ) -J(\Om )}{\eps },
	\end{equation*}
	where $\Om _\eps =\Phi_\eps (\Om )\mbox{ and }\Phi _\eps$ is given by~\eqref{trans}.
\end{definition}
%
In many of the physical situations, a shape functional $J(\Om )$ depends on the domain via the solution $u(\Om )$ of a boundary value problem defined in $\Om $. We call $u(\Om )$ as the shape function. The Eulerian derivative of such $J(\Om )$ depends on so called the shape derivative of $u(\Om )$. 
\begin{definition}[Material derivative]
	\label{defn:mat_deri}
	Let	$V$ be a vector field as given in~\eqref{trans}. For a shape function $u(\Om )\in W^{1,p}(\Om )$ with $1\leq p <\infty $, the material derivative $\dot{u}=\dot{u}(\Om ;V)$ in the direction of $V$ is defined as the following limit, if it exists:
	
	\noindent
	\begin{equation*}
		\dot{u}(\Om ;V)=\lim \limits _{\eps \searrow 0} \frac{1}{\eps } \Big(u(\Om _\eps )\circ \Phi _\eps - u(\Om ) \Big)\mbox{ in } W^{1,p}(\Om ),
	\end{equation*}
	where
	$\Om _\eps =\Phi _\eps (\Om ) \mbox{ and }\Phi _\eps$ is given by \eqref{trans}.
\end{definition}
\begin{definition}[Shape derivative]
	\label{defn:shape_der}
	Let	$V$ be a vector field as given in~\eqref{trans}. For a shape function $u(\Om )\in W^{1,p}(\Om )$, if the material derivative $\dot{u}(\Om ;V)$ exists and $\nabla u(\Om ) \cdot V\in L^p(\Om ),$ then  the shape derivative $u'=u'(\Om ;V)$ in the direction of $V$ is defined as
	\begin{equation}\label{eq:shape}
		u'(\Om ;V)=\dot{u}(\Om ;V)-\nabla u(\Om )\cdot V.
	\end{equation} 
\end{definition}
\begin{remark}
	We observe that, for a transformation $\Phi _\eps $ given in \eqref{trans}, for any compact set $K\subset \Om $ there exists $\eps _K>0$ such that $\Phi _\eps (K)\subset \Om $ for $\eps \leq \eps _K$. Therefore, the shape derivative of $u$ can be defined locally on every compact set as the derivative of $\eps \longmapsto \left. u(\Om _\eps )\right |_{K}$ in $K$. The material derivative of $u$ is the derivative of $\eps \longmapsto u(\Om _\eps )\circ \Phi _\eps $ in $\Om $. If the material derivative of $u$ exists then, from \cite[Corollary~5.2.3 or Lemma~5.2.7]{Henrot-Book}, the shape derivative exists and is given by \eqref{eq:shape}.
\end{remark}
\subsection{Shape derivative of the eigenfunctions and the eigenvalues}
Let $\Om \subset \RN$ be a bounded domain and let $\pa \Om =\Ga _D\cup \Ga _N$. We consider the following eigenvalue problem on $\Om :$
\begin{equation}\label{eigen1}
\left.
\begin{aligned} 
-\De u & = \tau u \mbox{ in } \Om ,\\
u &= 0 \mbox{ on } \Ga _D,\\
\frac{\pa u}{\pa n} &= 0  \mbox{ on } \Ga _N.
\end{aligned}
\right\}
\end{equation}
Next, we give the existence of the material derivative of the normalize eigenfunction $u(\Om )$ corresponding to the first eigenvalue $\tau _1(\Om )$ of~\eqref{eigen1}. See~\cite[Lemma 3.18, page 154]{Soko}, for the material derivative for a non-homogeneous problem with the mixed boundary conditions. For a vector field $V$ given in~\eqref{trans}, let $\Om _\eps $ be 
as given in \eqref{admisfamily} with the boundary $\pa \Om _\eps = \Ga _D^\eps \sqcup \Ga _N^\eps $, where $\Ga _i^\eps =\Phi _\eps (\Ga _i)$ for $i=D,N$. Let $(u_\eps , \tau _1(\eps ))$ be the first eigenpair of~\eqref{eigen1} in $\Om _\eps $ such that $\int _{\Om _\eps }u_\eps ^2\dx =1.$ We adapt the ideas of \cite[Theorem~5.3.2 and Theorem~5.7.2]{Henrot-Book} and show that both the material and the shape derivatives of $u(\Om )$ exist. 
\begin{proposition}\label{mat_der}
	Let $\tau _1(\Om )$ be the first eigenvalue of~\eqref{eigen1} in $\Om $ and let $u(\Om )$ be a corresponding eigenfunction with $\int _{\Om }u(\Om )^2\dx=1.$ Let	$V$ be a vector field as given in~\eqref{trans}. Then, 
	both the material derivative and the shape derivative of $u(\Om )$ exist, and also the Eulerian derivative of the shape functional $\tau _1(\Om )$ exists.
\end{proposition}
\begin{proof}
	Let $u_\eps \in \hd{\eps }$ be the eigenfunction corresponding to the first eigenvalue $\tau _1(\Om _\eps )$ of~\eqref{eigen1} on $\Om _\eps $ for $\eps \in [0,\eps _0)$ with the normalization $\int_{\Om _\eps }u_\eps ^2\dx=1$ , i.e.,

	\noindent
	\begin{equation*}
	\int _{\Om _\eps } \nabla u_\eps \cdot \nabla \varphi \dx -\tau _1(\Om _\eps ) \int _{\Om _\eps } u_\eps \varphi \dx =0 \quad \forall \varphi \in \hd{\eps }.
	\end{equation*}
	Let $v_\eps $ be $v_\eps =u_\eps \circ \Phi _\eps$ in $\Om $ for $\eps \in [0,\eps _0)$. For $\eps =0,$ we have $\Om _0 =\Om $ and $v_0=u_0=u(\Om )$. We observe that $v_\eps \in \HD{}$ and satisfies
	\begin{equation}\label{eqn:material1}
		\int _{\Om } \big(A(\eps )\nabla v_\eps \big)\cdot \nabla \psi \dx-\tau _1(\Om _\eps ) \int _{\Om }\ga (\eps ) v_\eps \psi \dx =0 \quad \forall \psi =\varphi \circ \Phi _\eps \in \HD{},
	\end{equation}
	where $A(\eps )=\ga (\eps )D\Phi _\eps ^{-1} ({D\Phi _\eps ^{-1}})^T,$ $\ga (\eps ) =\left |\det(D\Phi _\eps )\right | \mbox{ for }\eps \in [0,\eps _0 ).$ Thus, $v_\eps $ satisfies the following equation:
	
	\noindent
	\begin{equation*}
		-\mathrm{div}(A(\eps )\nabla v_\eps )=\tau _1(\Om _\eps )\ga (\eps ) v_\eps \mbox{ in }\Om .
	\end{equation*}
	Now, we consider the 
	function $\F :[0,\eps _0)\times \HD{}\times \R \longrightarrow H^{-1}(\Om )\times \R$ defined by
	
	\noindent
	\begin{equation*}
		\F (\eps ,v,\tau )=\left (-\mathrm{div}(A(\eps )\nabla v )-\tau \ga (\eps ) v , \int_{\Om }v^2 \dx -1\right ).
	\end{equation*}
	Then $\F$ is well-defined and continuous. We observe that 
	$\F(\eps ,v_\eps , \tau _1(\Om _\eps ))=(0,0)\ \text{(by~\eqref{eqn:material1})}$ 
	and
	
	\noindent
	\begin{equation*}
		D_{v,\tau _1}\F(0,u(\Om ), \tau _1(\Om ))(\phi ,\tau _1)
		=\left(-\De \phi -\tau _1(\Om )\phi -\tau _1u(\Om ) ,2\int_{\Om }u(\Om ) \phi \dx \right).
	\end{equation*}
	From~\cite[Lemma 5.7.3]{Henrot-Book}, it follows that $D_{v,\tau _1}\F(0,u , \tau _1(\Om ))$ is a bijection from $\HD{} \times \R$ onto $H^{-1}(\Om ) \times \R$. Since $D_{v,\tau _1}\F(0,u , \tau _1(\Om ))$ is a continuous linear map, it 
	is an isomorphism from $\HD{} \times \R$ onto $H^{-1}(\Om ) \times \R$. Thus, by the implicit function theorem, there exists a unique map $\eps \longmapsto (v(\eps ),\tau (\eps ))\in \HD{} \times \R $ of class $\C^1 $ on a neighborhood $\mathcal{U}$ of $0$ 
	into $\HD{} \times \R$ such that
 	
 	\noindent
 	\begin{equation*}
 		v(0)=u(\Om ),\quad \tau (0)=\tau _1(\Omega ), \quad \F(\eps ,v(\eps ),\tau _1(\eps ))=(0,0) \mbox{ for } \eps \in \mathcal{U}. 
 	\end{equation*}
 	By the uniqueness, it follows that $v(\eps )=v_\eps \mbox{ and }\tau (\eps )=\tau _1(\Om _\eps )$ for $\eps \in \mathcal{U}$. Since the map $\eps \longmapsto (v_\eps , \tau _1(\Om _\eps ))$ is differentiable at $\eps =0$, the Eulerian derivative of $\tau _1(\Om )$ and the material derivative of $u(\Om )$ in the direction of $V$ exist. Hence, from \cite[Corollary~5.2.3 or Lemma~5.2.7]{Henrot-Book}, the map $\eps \longmapsto u_\eps$ is differentiable at $\eps =0$, i.e., the shape derivative of $u(\Om )$ in the direction of $V$ exists.
\end{proof}
\begin{lemma}
	Let $u\in \HD{} $ be an eigenfunction corresponding to the first eigenvalue $\tau _1(\Om )$ of~\eqref{eigen1} in $\Om $. Let	$V$ be a vector field as given in~\eqref{trans}. Then, the material derivative $\dot{u}=\dot{u}(\Om ;V)$ of the eigenfunction $u$ is in $\HD{}$ and it is the unique solution of the following integral equation 
    \begin{equation}\label{eqn:material}
    	\int _{\Om } \nabla \dot{u}\cdot \nabla \varphi \dx-\tau _1\int _{\Om } \dot{u} \varphi \dx =\int _{\Om } \big(\tau _1'u+\tau _1u\,\mathrm{div}(V)\big)\varphi \dx -\int _{\Om } (A'(0)\nabla u)\cdot \nabla \varphi \dx \quad 
    	\forall \varphi \in \HD{}.
    \end{equation}
    %
    %
   where $A'(0)=\mathrm{div}(V)I-DV-{DV}^T$, $\tau _1'=d\tau _1(\Om ;V)$ the Eulerian derivative of $\tau _1(\Om )$ in the direction of $V$ at $\Om $.
\end{lemma}
\begin{proof}
%
    By Proposition~\ref{mat_der}, the map $\eps \longmapsto u_\eps$ is differentiable, i.e., the material derivative $\dot{u}(\Om ;V)$ exists and is in $\HD{}$. Therefore, by differentiating the equation~\eqref{eqn:material1} with respect to $\eps $ at $\eps =0$, we get that
    
    \noindent
    \begin{equation*}
    	\int _\Om (A'(0)\nabla u)\cdot \nabla \varphi \dx + \int _\Om \nabla \dot{u}\cdot \nabla \varphi \dx - \int _\Om \dot{(\tau _1u)} \varphi \dx -\tau _1\int _\Om u\ga '(0)\varphi \dx =0\quad \forall \varphi \in \HD{},
    \end{equation*}
    with $A'(0)=\mathrm{div}(V)I-DV-{DV}^T$.
    Noting that $\ga '(0)={\rm div}(V)$ and using the product rule for the material derivative, we obtain~\eqref{eqn:material}.
\end{proof}
\begin{remark}
	The material derivative of an eigenfunction $u(\Om )\in \HD{}$ corresponding to the first eigenvalue $\tau _1(\Om )$ of~\eqref{eigen1} in $\Om $, is the unique weak solution of the following boundary value problem:
	
	\noi
	\begin{equation*}
		\left.
		\begin{aligned}
		-\De \dot{u}-\tau _1(\Om ) \dot{u}&=f(u)\text{ in } \Om ,\\
		\dot{u} &=0\text{ on }\Ga _D,\\
		\frac{\pa \dot{u}}{\pa n} &= -(A'(0)\nabla u)\cdot n \text{ on } \Ga _N,
		\end{aligned}
		\right\}
	\end{equation*}
	with $f(u)=(\tau _1'+\tau _1\,\mathrm{div}(V))u+\mathrm{div}(A'(0)\nabla u).$
\end{remark}

\par From the definition of the shape derivative of $u$, we have the following lemma.
\begin{lemma}\label{shape:lemma0}
	Let	$V$ be a vector field as given in~\eqref{trans}. The shape derivative $u'=u'(\Om ;V)$ of the eigenfunction $u(\Om )$ corresponding to the first eigenvalue $\tau _1(\Om )$ of~\eqref{eigen1} in $\Om $, is the unique weak solution of the following boundary value problem:
	
	\noindent
    \begin{equation*}
		\left.
		\begin{aligned}
			-\De u'&=\tau _1(\Om ) u'+ \tau _1' u \text{ in } \Om ,\\
			u' &= -\frac{\pa u}{\pa n}\, (V\cdot n) \text{ on } \Ga _D,\\
			\frac{\pa u'} {\pa n} &= -\big(A'(0)\nabla u+\nabla (\nabla u \cdot V)\big)\cdot n 
			\text{ on } \Ga _N.
	    \end{aligned}
	   \right\}
    \end{equation*}
    %
\end{lemma}
\begin{proof}
	Using the definition of $u'=\dot{u}-\nabla u\cdot V$ in $\Om $, we rewrite \eqref{eqn:material} as
    %
    \begin{equation}\label{eqn:material-1}
    	\int _\Om \nabla u'\cdot \nabla \varphi \dx -\tau _1\int _\Om u'\varphi \dx -\tau _1'\int _\Om u\varphi \dx
    	=-\int _\Om \big( \nabla (\nabla u\cdot V)- (A'(0)\nabla u)\big)\cdot \nabla \varphi \dx +\tau _1\int _\Om \mathrm{div}(uV)\varphi \dx.
    \end{equation}
    %
    %
    %
    We claim that
    
    \noindent
    \begin{equation*}\label{shape:eqn1}
        \int _\Om \big(\nabla (\nabla u\cdot V)+ (A'(0)\nabla u) \big) \cdot \nabla \varphi \dx-\tau _1\int _\Om \mathrm{div} (uV)\varphi \dx =0 \quad \forall \varphi \in \D (\Om ).
    \end{equation*}
    For this, we define $f(u,V)\dq A'(0)\nabla u +\nabla (\nabla u \cdot V)-\De u V$ in $\Om $. It is easy to verify that $\mathrm{div}(f(u,V))=0$. 
    Now, for $\varphi \in \D(\Om )$, we have
    \begin{eqnarray*}
    	\int _\Om \big(\nabla (\nabla u\cdot V)\!
    	+\! A'(0)\nabla u\big)\cdot \nabla \varphi \dx
    	&=& \int _\Om \Big(f(u,V)+\De u V\Big)\cdot \nabla \varphi \dx\\
    	&=& -\int _\Om \mathrm{div}(f(u,V))\varphi \dx
    	+\int _{\pa \Om }\varphi f(u,V)\cdot n\dS -\tau _1\int _\Om u V\cdot \nabla \varphi \dx\\
    	&=& -0+0-\tau _1\int _\Om u V\cdot \nabla \varphi \dx\\
    	&=&\tau _1\int _\Om \mathrm{div}(uV)\varphi \dx
    	-\tau _1\int _{\pa \Om }\varphi u V\cdot n \dS
    	=\tau _1\int _\Om \mathrm{div}(uV)\varphi \dx. 
    \end{eqnarray*}
    The last equality follows from the integration by parts.
    Therefore, \eqref{eqn:material-1} gives:
    
    \noindent
    \begin{equation*}
    	\int _\Om \nabla u'\cdot \nabla \varphi \dx -\tau _1\int _\Om u'\varphi \dx -\tau _1'\int _\Om u\varphi \dx=0 \quad \forall \varphi \in \D (\Om ).
    \end{equation*}
    Now, using the integration by parts we get
       
    \noindent
    \begin{equation*}
    	-\int _\Om \De u' \varphi \dx - \tau _1\int _\Om u' \varphi \dx -\tau _1' \int _\Om u \varphi \dx=0 \quad \forall \varphi \in \D (\Om ),
    \end{equation*}
    %
    and hence

    \noindent
    \begin{equation*}
    	-\De u'=\tau _1u' +\tau _1' u \text{ in }\Om .
    \end{equation*}
    Since $u=0$ on $\Ga _D$, we have $\dot{u}=0$ on $\Ga _D$ and also $\nabla u = \displaystyle \frac{\pa u}{\pa n}n$. Now, from the definition of the shape derivative $u'$ we get that
    
    \noindent
    \begin{equation*}
    	u'=\dot{u}-\frac{\pa u}{\pa n}V\cdot n=-\frac{\pa u} {\pa n}V\cdot n \text{ on }\Ga _D.
    \end{equation*}
    On $\Ga _N$
    %
    
    \noindent
    \begin{equation*}
    	\frac{\pa u'}{\pa n}
    	= \frac{\pa \dot{u}}{\pa n}- \nabla (\nabla u\cdot V) \cdot n
    	=-\big(A'(0)\nabla u +\nabla (\nabla u \cdot V)\big)\cdot n.
    \end{equation*}
    This completes the proof.
    %
    %
\end{proof}
\par Next, we compute the Eulerian derivative of the first eigenvalue $\tau _1(\Om )$ of~\eqref{eigen1}.
\begin{theorem}\label{thm:shape}
	Let $V$ be a vector field given in~\eqref{trans}. If $\tau _1(\Om )$ is the first eigenvalue of~\eqref{eigen1} and $u$ is an eigenfunction corresponding to $\tau _1(\Om )$ with the normalization $\int _\Om u^2 \dx=1$, then the Eulerian derivative of $\tau _1(\Om )$ in the direction of $V$ is given by
	\begin{equation}\label{shape formula}
		d\tau(\Om ;V)=\int _{\Ga _N} \left(\md{\nabla u}^2-\tau _1(\Om ) u^2\right)\, (V\cdot n) \dS - \int _{\Ga _D} \left(\frac{\partial u}{\partial n}\right)^2\, (V\cdot n) \dS.
	\end{equation}
\end{theorem}
\begin{proof}
	Firstly, we recall (from \cite[Theorem 3.3]{Simon80}) that, if a shape functional $J(\Om )$ is given by $J(\Om )=\int _\Om C(u(\Om ))\dx$ under the assumptions that:  the operator $C:W^{1,p}(\Om )\longrightarrow L^1(\Om )$ is differentiable, and both the functions $u(\Om )\in W^{1,p}(\Om )$ and $C(u(\Om ))\in W^{1,1}(\Om )$ have the material derivative in the direction of $V$, then the Eulerian derivative of $J$ in the direction of $V$ exists and is given by 
	\begin{equation}\label{shape derivative}
	\mathrm{d}J(\Om ;V)=\int _\Om \frac{\pa C}{\pa u}(u')\dx + \int _{\pa \Om }C(u)(V\cdot n)\dS.
	\end{equation}
	Let $u$ be an eigenfunction corresponding to $\tau _1(\Om )$ with the normalization $\int _\Om u^2 \dx=1$. From Proposition~\ref{mat_der}, the material derivatives of $u(\Om )$ and $C(u(\Om ))$ exist, and therefore the Eulerian derivative of the shape functionals $\int _\Om |\nabla u|^2 \dx $ and $\int _\Om u^2 \dx $ exist. 
	By taking shape derivative of $\int _{\Om }u^2\dx =1$ in the direction of $V$ on both sides, and using \eqref{shape derivative} together with the boundary conditions yield
	\begin{equation}\label{shape:l2}
        2 \int _{\Om } u u'\dx +\int _{\Ga _N} u^2 \, (V\cdot n) \dS =0.
    \end{equation}
    Notice that, $\tau _1(\Om )=\int _\Om |\nabla u|^2 \dx $. Thus, again by~\eqref{shape derivative}, we get
    
    \noindent
    \begin{equation*}
    	d\tau _1(\Om ;V)
    	=2 \int _{\Om }\nabla u\cdot \nabla u'\dx
    	+ \int _{\pa \Om}\md{\nabla u}^2\, (V\cdot n) \dS.
    \end{equation*}
    Since $u$ is constant on $\Ga_D$, we have
    $\nabla u =\displaystyle \frac{\pa u}{\pa n}n$
    and
    $\md{\nabla u}= \md{\displaystyle \frac{\partial u}{\partial n}}$
    on $\Ga _D$. Therefore 
    \begin{equation}\label{shape_der1}
    	d\tau _1(\Om ;V)
    	=2 \int _{\Om}\nabla u\cdot \nabla u'\dx
    	+\int_{\Ga _D} \left(\frac{\pa u}{\pa n}\right)^2\, (V\cdot n) \dS +\int _{\Ga _N}\md{\nabla u}^2\, (V\cdot n) \dS.
  \end{equation}
    By multiplying the equation $-\De u=\tau _1(\Om )u$ with $u'$ and then using the integration by parts we obtain 
    \begin{equation}\label{shape_der2}
    	\int _{\Om } \nabla u\cdot \nabla u'\dx- \int _{\Ga _D} \frac{\pa u}{\pa n} u^{\prime } \dS = \tau _1(\Om ) \int _{\Om }u u'\dx.
    \end{equation}
    By combining~\eqref{shape:l2}-\eqref{shape_der2}, we get:
    %
%
    
    \noindent
    \begin{equation*}
    	d\tau _1(\Om ;V)= \int _{\Ga _N} \left(\md{\nabla u}^2-\tau _1(\Om ) u^2\right)\, (V\cdot n) \dS - \int _{\Ga _D} \left(\frac{\pa u}{\pa n}\right)^2\, (V\cdot n) \dS. \qedhere
    \end{equation*}
\end{proof}
For a similar shape derivative formula for the Dirichlet eigenvalue, we refer to \cite{Anoop18, Zivari08}.
\section{Maximum principle at corners}\label{MPatCorners}
Next we state the a version of Hopf's lemma that can be used to to determine the tangential derivative of superharmonic function  at the corners. The same result can be be obtained from from
\cite[Lemma A.2]{Pedro-Tobis} for the annular domains. We are giving a simpler proof for the same.  
Recall that, $\H_0$ is the set of all closed half spaces (polarizers) of $\RN$ and $\H ^* = \{H\in \H_0 : -e_1\in H\}$. 
%
Let $H\in \H^*$ be a fixed polarizer, then any set $A\subseteq \RN,$ we denote $A \cap H$ by $A^+$ and $A\cap H\cm $ by $A^-$. 
\begin{proposition}\label{propo:Hopf}
	Let $\Om =B_{R}(0)\setminus \overline{B_r (s e_1)}$ be an annular domain, for $0<r<R<\infty $ and $0\leq s<R-r$. Let $H\in \H ^*$ and let $\Om ^\pm $ be given as above. 
	Let
	$w\in \C^2 (\Om ^-)\cap \C^1 (\overline{\Om ^-})$ be a solution of the following problem:
	
	\noindent
	\begin{equation*}
		\left .
		\begin{aligned}
			-\De w &\geq 0 \mbox{ in } \Om ^-,\\
			w& \geq 0 \mbox{ on } (\overline{\Om }\cap \pa H)\cup(\pa B_{r}(a))^-,\\
			\frac{\pa w}{\pa n} &\geq  0  \mbox{ on } (\pa B_{R}(0))^- ,
		\end{aligned}
		\right \}
	\end{equation*}
	where~$n$ is the unit outward normal to $\Om ^-$. If $w\not \equiv 0$ in $\Om ^{-}$, then $w>0$ on  $\Om ^-\cup (\pa B_{R}(0))^-$ and for an outward normal $h$ to $H$ we have 
	
	\noindent
	\begin{equation*}
		\frac{\pa w}{\pa h } (x_0)>0 \text{ for } x_0\in \pa B_R(0) \cap \pa H \text{ with } w(x_0)=0.
	\end{equation*}
\end{proposition}
%
\begin{proof}
	Since $w\not \equiv 0$ in $\Om ^-$, by the strong maximum principle, $w>0$ in $\Om ^-.$ Since $\frac{\pa w}{\pa n}\geq 0$, by Hopf's lemma, the minimum value of $w$ can not be attained on $(\pa B_R(0))^-$, and therefore $w>0$ on $(\pa B_R(0))^-$. Hence $w>0$ on $\Om ^-\cup (\pa B_{R}(0))^-.$ 
	Let $h$ be an outward normal to $H$. It remains to show that $\frac{\pa w}{\pa h}(x_0)>0$ at $x_0\in \pa B_R(0) \cap \pa H$ with $w(x_0)=0.$ 
	\medskip \\
	\noindent
	\emph{For $x_0\in \pa B_R(0) \cap \pa H$ with $w(x_0)=0$:}
	By an appropriate rotation of the coordinate frame, we can assume that the hyperplane $\pa H$ is given by $\{x\in \RN : x_d=0\}$ and $h=e_d$. Now, we scale the domain $\Om $ so that $R=1$. By the symmetry of $\pa \Om \cap \pa H$, it is enough to choose the corner point  $x_0$ to be $e_1$. Let $\overline{x}=e_1+e_d=(1,0,\dots ,0,1)$. 
	We consider the function $\varphi : \RN \longrightarrow \R$ defined by
	
	\noindent
	\begin{equation*}
		\varphi (x)=
		\left|\frac{x}{|x|^2}-\overline{x}\right|^{-\al }-1
		,
	\end{equation*}
	where $\al >0$ will be fixed later. Then $\varphi \equiv 0$ on $\pa B_1(\overline{x})$ and 
	$\varphi (e_1)=0.$ By direct computation, we obtain
	
	\noindent
	\begin{equation*}
		\frac{\pa \varphi }{\pa x_i}(x)=-\al \left|\frac{x}{|x|^2}-\overline{x} \right|^{-\al -2}\sum\limits_{j=1}^{d}\left(\frac{x_j}{|x|^2}-\overline{x}_j\right)\left( \frac{\de _{ij}}{|x|^2}-\frac{2x_i x_j}{|x|^4} \right) \mbox{ for }i=1,2,\dots ,d.
	\end{equation*}
	Therefore, we get
	
	\noindent
	\begin{equation*}
		\frac{\pa \varphi }{\pa x_d}(e_1)=\al >0.
	\end{equation*}
	Again, by direct computation
	
	\noindent
	\begin{equation*}
		\frac{\pa \varphi }{\pa x_i^2}(e_1)=\left. \al (\al +2)(\de _{id})^2-4\al \de _{id}x_i -\al \sum \limits_{j=1}^{d}\left(\de _{ij}-2x_i x_j \right)^2\right|_{x=e_1} \mbox{ for }i=1,2,\dots ,d,
	\end{equation*}
	and hence
	
	\noindent
	\begin{equation*}
		\De \varphi (e_1)
		=\al (\al +2)\sum\limits _{i=1}^d (\de_{id})^2-4\al \sum\limits_{i=1}^d\de_{id} x_i-\al \sum\limits_{i,j=1}^d(\de _{ij}-2x_i x_j)^2=\al (\al +2-d)>0,
	\end{equation*}
	for $\al >d-2.$
	Note that, for $x\in \pa B_1(0)$ we have $n(x)=x$ and hence
	
	\noindent
	\begin{equation*}
		\frac{\pa \varphi }{\pa n}(x) = \al  |x-\overline{x}|^{-\al -2}(1- x\cdot \overline{x}) \mbox{ for }x\in \pa B_1(0). 
	\end{equation*}
	For $x\in \overline{B_1 (\overline{x})}\cap \pa B_1(0)$, we have
	$1 \geq |x-\overline{x}|^2
	=|x|^2-2x\cdot \overline{x} + |\overline{x}|^2=3-2x\cdot \overline{x}.$
	Hence
	
	\noindent
	\begin{equation*}
		\frac{\pa \varphi }{\pa n}(x)\leq 0  \mbox{ for } x\in \overline{B_1 (\overline{x})}\cap \pa B_1(0).
	\end{equation*}
	Let $\de >0$ be small enough such that $\overline{B_\de (e_1)}\cap \pa B_r(a)=\emptyset $ and 
	
	\noindent
	\begin{equation*}
		\De \varphi >0 \mbox{ in } B_\de (e_1). 
	\end{equation*}
	Now, we consider the set $D\dq B_\de (e_1)\cap B_1(\overline{x})\cap \Om ^-$. Since, $w>0$ on $\pa B_\de (e_1)\cap B_1(\overline{x})\cap \Om ^-$, there exists $\eps >0$ such that
	
	\noindent
	\begin{equation*}
		\inf \{ w (x) :x\in \pa B_\de (e_1)\cap B_1(\overline{x})\cap \Om ^- \}\geq \eps \sup \{\varphi (x): x\in \pa B_\de (e_1)\cap B_1(\overline{x})\cap \Om ^-\}, 
	\end{equation*}
	and hence $w\geq \eps \varphi$ on $\pa B_\de (e_1)\cap B_1(\overline{x})\cap \Om ^-$. The function $\wide{w}:\overline{D}\longrightarrow \R$ defined by $\wide{w}=w-\eps \varphi $ satisfies
	
	\noindent
	\begin{equation*}
		\left .
		\begin{aligned}
			-\De \wide{w} &> 0 \mbox{ in } D,\\
			\wide{w}\geq 0  \mbox{ on } \pa B_\de (e_1)\cap B_1(\overline{x})\cap \Om ^-, &\;\; \wide{w} > 0 \mbox{ on } B_\de (e_1)\cap \pa B_1(\overline{x})\cap \Om ^-,\\
			\frac{\pa \wide{w}}{\pa n} &\geq 0  \mbox{ on } B_\de (e_1)\cap B_1(\overline{x})\cap \pa \Om ^-.
		\end{aligned}
		\right \}
	\end{equation*}
	By the similar arguments as above, $\wide{w}\geq 0$ in $D$, and therefore by the strong maximum principle $\wide{w} >0\mbox{ in } D$. For any $\nu $ such that $e_1+t\nu \in D$ for small $t$, we have $\frac{\pa \wide{w}}{\pa \nu }(e_1)\geq 0$, since $\wide{w}(e_1)=0.$ 
	Then, by smoothness of $\wide{w}$, we get
	
	\noindent
	\begin{equation*}
		\frac{\pa \wide{w}}{\pa x_d}(e_1) 
		\geq 0.
	\end{equation*}
	Hence
	
	\noindent
	\begin{equation*}
		\frac{\pa w}{\pa x_d}(e_1)\geq \eps \frac{\pa \varphi }{\pa x_d}(e_1)>0. \qedhere 
	\end{equation*}
\end{proof}
We would like to stress that the above version of Hopf's lemma gives the sign of G\^{a}teau derivative along the tangential directions at point on the boundary where boundary intersect transversally.
	%
	
\end{document}